\documentclass[11pt, reqno]{amsart}

\usepackage{amsmath,amssymb,latexsym}
\usepackage{color}
\usepackage{xcolor}
\usepackage{graphicx}
\usepackage{cite}
\usepackage[colorlinks=true]{hyperref}
\hypersetup{urlcolor=blue, citecolor=red, linkcolor=blue}

\newcommand{\R}{\mathbb{R}}

\DeclareMathOperator{\supp}{supp}

\newtheorem{theorem}{Theorem}[section]
\newtheorem{lemma}{Lemma}

\newtheorem{corollary}{Corollary}
\newtheorem*{main-theorem}{Main Theorem}
\newtheorem*{remark*}{Remark}
\newtheorem*{lemma*}{Lemma A.1}

\numberwithin{equation}{section}

\begin{document}

	\title[Generalized Fornberg-Whitham equation]{Refined wave breaking for the generalized Fornberg-Whitham equation}

	\author{Jean-Claude Saut, and Yuexun Wang} 
	\address{Université Paris-Saclay, CNRS, Laboratoire de Mathématiques d'Orsay, 91405 Orsay, France}
	\email{jean-claude.saut@universite-paris-saclay.fr}

	\address{School of Mathematics and Statistics, Lanzhou University, 730000 Lanzhou,
		People's Republic of China}
	\email{yuexunwang@lzu.edu.cn}


	\begin{abstract}

This paper considers a class of non-local equations that are weakly dispersive perturbations of the inviscid Burgers equation, which includes the Fornberg-Whitham equation as a special case. We precise the known results on finite time blow-up (shock formation) by constructing a blowup solution which displays a `shock-like' singularity (called wave breaking) at one single point. Moreover, this solution converges asymptotically in the self-similar variables to a stable self-similar solution of the inviscid Burgers equation, and also possesses a H\"{o}lder $C^{1/3}$ regularity at the blowup point.

	\end{abstract}
	
\maketitle

	\section{Introduction}

This paper considers the following dispersive perturbations of the inviscid Burgers equation:
\begin{equation}\label{dpb}
u_t+uu_x-Lu_x=0,
\end{equation}
where $(x, t)\in \mathbb{R}\times \mathbb{R}$ and the Fourier multiplier operator $L$ is defined by 
\begin{equation*}
\begin{aligned}
\widehat{Lf}(\xi)=p(\xi)\hat{f}(\xi)
\end{aligned}
\end{equation*}
for any $f\in \mathcal{S}(\mathbb{R})$. 

\eqref{dpb} contains some typical examples such as the fractional KdV (fKdV) equation, the Whitham equation and so on. More precisely, the fKdV equation is usually written as \eqref{dpb} with $p(\xi)=|\xi|^{\alpha}$ for $\alpha\in [-1,1)\setminus\{0\}$ (the case $\alpha=0$ recovers the inviscid Burgers equation after a suitable change of
variable) \footnote{When \(\alpha=-1\), \eqref{dpb} is referred as the Burgers-Hilbert (BH) equation introduced in \cite{MR2599457} to describe
the dynamics of vortex patches.}. The dynamics of solutions to \eqref{dpb} for positive and negative $\alpha$'s are quite different. 
The numerical simulations \cite{MR3317254} suggest that the fKdV equation is globally well-posed in the energy space for $\alpha\in (1/2,1)$, which has been proven recently \cite{MR3906854} for $\alpha\in (6/7,1)$ leaving the remaining case  open. When $\alpha\in [-1,1/2)\setminus\{0\}$, \cite{MR3317254} also suggests that the solutions to the fKdV equation  with somehow large initial data will develop singularities but with different blowup mechanisms depending on the values of $\alpha$ in $(1/3,1/2]$, $(0,1/3]$ \footnote{The blowup for $\alpha$ in $(1/3,1/2]$ and $(0,1/3]$ is completely open.} and 
$[-1,0)$. When $\alpha\in[-1,0)$, the blowup    
was earlier proven in \cite{MR2727172} in the sense of $C^{1+\delta}$ norm, and further confirmed as a wave breaking (the gradient of solution blows up whereas the solution  itself remains bounded) in a series of works \cite{MR4409228, MR3682673, MR4752990, MR4321245} via different methods. It should be pointed out that 
\cite{MR4321245} uses the modulated self-similar blowup technique developed in earlier works \cite{MR2470260, MR0784476, MR1409662, MR1427848}, and especially in  the recent work \cite{MR4612576}, to construct an open set of smooth initial data yielding solutions with gradient blowup  to the BH equation with the ground state self-similar solutions to the inviscid Burgers equation as the blowup profiles. The asymptotic convergence of the blowup solutions is also discussed in \cite{MR4321245}. \cite{MR4752990} establishes the gradient blowup for the full region $\alpha\in[-1,0)$, using again the smooth self-similar solutions to the inviscid Burgers equation as the blowup profiles, but  allowing for the use of excited states (in the region $\alpha\in [-1/3,0)$). (see also \cite{BI, MR4468858, Yang} for the use of excited states as the blowup profiles to study blowup in different contexts of PDEs). The argument in  \cite{MR4752990} also works for the dissipative perturbations of the inviscid Burgers equation (see also an independent work \cite{MR4664477} which is  closer to \cite{ MR4321245} in methodology). 

The Whitham equation
($p(\xi)=\sqrt{\tanh \xi/\xi}$ in \eqref{dpb}) was originally proposed in \cite{MR0671107} as an alternative to the KdV equation, by keeping the exact dispersion of the water waves system in finite depth.  Due to the very different behavior of the dispersion at low and high frequencies, the Whitham equation displays rich dynamics such as highest waves \cite{MR4002168}, wave breaking \cite{MR4409228, MR3682673, MR4752990}
and long wave limit to the KdV equation \cite{MR3763731}. This long wave limit implies that 
the Whitham equation possesses solitary waves \cite{MR2979975}, although the fKdV equation does not posses any solitary waves when $\alpha\in [-1,1/3]\setminus\{0\}$ \cite{MR3188389} (Note that the dispersion of the Whitham equation is the same as that of the fKdV equation with $\alpha=-1/2$ for high frequencies).

	
Another typical example of \eqref{dpb} is the following  
	\begin{equation}\label{eq:1.1}
	\begin{aligned}
	u_t+uu_x-K_\alpha u_x =0,
	\end{aligned}
	\end{equation}
where  $K_\alpha = (I- \partial_{x}^{2})^{(\alpha-1)/2}$ and thus its Fourier multiplier has the expression $p(\xi)=(1+|\xi|^2)^{(\alpha-1)/2}$. The advantage of considering such a dispersive perturbation instead of the dispersion of the aforementioned fKdV equation,  is that one always gets a smooth phase velocity and that the dispersion can have arbitrarily small order. Another advantage is that in the long wave limit one obtains formally the KdV equation and thus one keeps a dispersive regime, as it is the case for the Whitham equation. 

When $\alpha=-1$, \eqref{eq:1.1} corresponds to the Fornberg-Whitham equation introduced in \cite{MR0671107, MR1699025}, whose wave breaking was first studied in \cite{MR1668586} and then sharpened in \cite{MR3710682, MR3809008, MR3552557, MR4270781,  MR4510790, MR4561680, MR3825177, MR4333381}. 
Recently,  it is shown that \eqref{eq:1.1} develops wave breaking for $\alpha\in [0,1)$ in \cite{MR4752990}, and $\alpha<0$ in \cite{MR4766878} by different methods, respectively. 
We refer to \cite{MR4270781} for a review of various issues concerning the Cauchy problem for the Fornberg-Whitham equation.


In this paper, we aim to show a  refined wave breaking for  \eqref{eq:1.1} in the region $\alpha<0$ in the sense of a precise description of  the blowup time and location, of the regularity and the asymptotic convergence of the blowup solutions. Our results include the Fornberg-Whitham equation as a special case.  Since \cite{MR1668586, MR3710682, MR3809008, MR3552557, MR4270781, MR4510790, MR4561680, MR3825177, MR4333381} consider quite general initial data, very accurate blowup information seems hard to be expected. The present work will adopt  the modulated self-similar blowup technique as aforementioned (the argument is closer to \cite{MR4321245}), and use the ground state self-similar solutions to the inviscid Burgers equation as the blowup profiles to show refined wave breaking for  \eqref{eq:1.1} in the region $\alpha<0$.

Comparing with the case $\alpha\in [0,1)$,  the dispersive effect of \eqref{eq:1.1} is much weaker when $\alpha<0$, so that  \eqref{eq:1.1} in this case seems to be more `hyperbolic' when the dispersion fights with the nonlinearity for large solutions. In fact, our results will confirm this fact in the sense that much less hypotheses on initial data can lead to wave breaking. More precisely, contrary to  \cite{MR4321245}, we do not  
impose any decay bootstrap assumptions for $\partial_X^2U$ on the middle field $h\leq  |X|\leq \frac{1}{2}e^{\frac{3s}{2}}$ and the far field $\frac{1}{2}e^{\frac{3s}{2}}\leq  |X|<\infty$. In \cite{MR4321245}, in order to prove the boundedness of $\partial_X^3U$, one needs a certain spatial decay $\langle X\rangle^{-\frac{2}{3}}$ of $\partial_X^2U$ on the middle field $h\leq  |X|\leq \frac{1}{2}e^{\frac{3s}{2}}$, and some temporal decay $e^{-s}$ on the far field $\frac{1}{2}e^{\frac{3s}{2}}\leq  |X|<\infty$ to kill the growth in the forcing term,  difficulties  coming mainly from the Hilbert transform. In our situation (i.e., $\alpha<0$), since the Fourier multiplier $p(\xi)$ is inhomogeneous in $\xi$ (which leads to some new difficulties, see below), the kernel $K_\alpha\partial_x$ is split into a low frequency operator $\mathcal{L}$ and a high frequency operator $\mathcal{H}$ depending on the scale of the self-similar transformation, leading to quite different behaviors: 
the higher-order derivative of $\partial_X^l\mathcal{L}$ contributes to a good temporal decay, while   
$\partial_X^l\mathcal{H}$ does not; despite of that, $\mathcal{H}$ has a good integral form that is bounded from $L^\infty$ to $L^\infty$,  providing effective bounds  
on $\partial_X^l\mathcal{H}$ for our estimates.

The paper is organized as follows. In Section \ref{PM}, we first reformulate \eqref{eq:1.1} into a   
self-similar form \eqref{eq:2.4} by carrying out the modulated self-similar transformation, and next state the assumptions on the initial data, and then state the main results. Section \ref{bootstraps} lists the bootstrap assumptions on the solutions and the dynamic modulation variables. In Section \ref{l2 bounds}, we derive some $L^2$ bounds which will be used to control the $L^\infty$ bounds of the operators $\mathcal{L}$ and $\mathcal{H}$. Section \ref{LFM} is devoted to the study of 
some analytic properties of the operators $\mathcal{H}$ and $\mathcal{L}$.  In Section \ref{PE}, 
we close the bootstrap assumptions listed in section \ref{bootstraps}. We finally prove Theorem \ref{thm2.1} in Section \ref{Proof of main result}.

\section{Preliminaries and Main results}\label{PM}

\subsection{Notation and conventions.}
Let \(\mathcal{F}(g)\) or \(\widehat{g}\) be the Fourier transform of a Schwartz function \(g\in \mathcal{S}(\mathbb{R})\) whose formula is given by 
\begin{equation*}
\begin{aligned}
\mathcal{F}(g)(\xi)=\widehat{g}(\xi):=\frac{1}{\sqrt{2\pi}}\int_{\mathbb{R}}g(x)e^{-\mathrm{i}x\xi}\,d x
\end{aligned}
\end{equation*}
with inverse
\begin{equation*}
\begin{aligned}
\mathcal{F}^{-1}(g)(x)=\frac{1}{\sqrt{2\pi}}\int_{\mathbb{R}}g(\xi)e^{\mathrm{i}x\xi}\, d \xi,
\end{aligned}
\end{equation*}
and by \(m(\partial_x)\) the Fourier multiplier with symbol \(m\) via the relation 
\begin{equation*}
\begin{aligned}
\mathcal{F}\big(m(\partial_x)g\big)(\xi)=m(\mathrm{i}\xi)\widehat{g}(\xi).
\end{aligned}
\end{equation*}

Take \(\varphi\in C_0^\infty(\R)\) satisfying \(\varphi(\xi)=1\) for \(|\xi|\leq 1\) and \(\varphi(\xi)=0\) when \(|\xi|>2\), and
let 
\begin{equation*}
\begin{aligned}
\psi(\xi)=\varphi(\xi)-\varphi(2\xi),\quad \psi_j(\xi)=\psi(2^{-j}\xi),\quad \varphi_j(\xi)=\varphi(2^{-j}\xi),
\end{aligned}
\end{equation*}
we then may define the  
Littlewood-Paley projections \(P_j,P_{\leq j},P_{> j}\)  via 
\begin{equation*}
\begin{aligned}
\widehat{P_jg}(\xi)=\psi_j(\xi)\widehat{g}(\xi),\quad \widehat{P_{\leq j}g}(\xi)=\varphi_j(\xi)\widehat{g}(\xi),\quad P_{> j}=1-P_{\leq j}.
\end{aligned}
\end{equation*}

The notation \(C\)  always denotes a nonnegative universal constant which may be different from line to line but is
independent of the parameters involved. Otherwise, we will specify it by  the notation \(C(a,b,\dots)\).
We write \(g\lesssim h\) (\(g\gtrsim h\)) when \(g\leq  Ch\) (\(g\geq  Ch\)), and \(g \sim h\) when \(g \lesssim h \lesssim g\).
We also write \(\sqrt{1+x^2}=\langle x\rangle\) for simplicity.	
	
\subsection{The self-similar transformation}

Since \eqref{eq:1.1} is time translation invariant, one can choose $t_0=-\varepsilon$ as the initial time. 
It is standard to show that \eqref{eq:1.1} is locally-in-time well-posed in $C([-\varepsilon, T_*), H^5(\mathbb{R}))$ for initial data $u_0(\cdot)=u(\cdot, t_0)\in H^5(\mathbb{R})$ (see for instance \cite{MR533234}). For convenience, we assume that $T_*$ is the maximal time of existence in this class of $u$ throughout the rest of this paper. 

We first introduce three dynamic modulation variables 
$$\xi,\ \tau, \ \kappa: [-\varepsilon ,T_{*}]\to \mathbb{R},$$
 which are parameters that will be used to control the location, time,  and wave amplitude of the wave breaking, and then set
	\begin{equation}\label{eq:2.1}
	\xi(-\varepsilon)=0,\  \tau(-\varepsilon)=0, \  \kappa(-\varepsilon)=\kappa_{0},
	\end{equation}
and
     \begin{equation}\label{eq:2.2}
	\xi (T_{*})=x_{*},\ \tau (T_{*})=T_{*},
	\end{equation}
where $T_{*}$ and  $x_*$ are the blowup time and location, respectively. 
Later, we will show that $T_{*}$ is the unique fixed point of $\tau$.

Next, we introduce a self-similar transformation
	\begin{equation}\label{eq:2.3}
     \begin{aligned}
      s(t)=-\log \big(\tau(t)-t\big),\
	X(x,t)=\frac{x-\xi(t)}{\big(\tau(t)-t\big)^{\frac{3}{2}}},
     \end{aligned}
	\end{equation}
and a self-similar ansatz
	\begin{equation}\label{eq:2.4}
	u(x,t)=e^{-\frac{s}{2}}U(X,s)+\kappa(t). 
	\end{equation}	
Now, we insert \eqref{eq:2.3} and \eqref{eq:2.4} into \eqref{eq:1.1} to deduce
	\begin{equation}\label{eq:2.5-0}
     \begin{aligned}
	&\bigg(\partial_{s}-\frac{1}{2}\bigg) U+ \left[\frac{3}{2}X+\frac{1}{1-\dot{\tau}}\big(U+e^{\frac{s}{2}}(\kappa-\dot{\xi})\big)\right] \partial_{X} U\\
&=-\frac{1}{1-\dot{\tau}}e^{-\frac{s}{2}} \dot{\kappa}+\frac{1}{1-\dot{\tau}}e^{\frac{s}{2}}(I-e^{3 s}\partial_X^2)^{\frac{\alpha-1}{2}} \partial_XU.
      \end{aligned}
	\end{equation}	
Since the Fourier multiplier of $(I-e^{3 s}\partial_X^2)^{\frac{\alpha-1}{2}}$ is inhomogeneous, we decompose it into a high frequency part and a low frequency part as \cite{MR4752990} to get 
\begin{equation}\label{eq:2.5}
     \begin{aligned}
	&\bigg(\partial_{s}-\frac{1}{2}\bigg) U+ \left[\frac{3}{2}X+\frac{1}{1-\dot{\tau}}\big(U+e^{\frac{s}{2}}(\kappa-\dot{\xi})\big)\right] \partial_{X} U\\
&=:-\frac{1}{1-\dot{\tau}}e^{-\frac{s}{2}} \dot{\kappa}+\frac{1}{1-\dot{\tau}}e^{-s}\big(\mathcal{H}(U)+\mathcal{L}(U)\big),
      \end{aligned}
	\end{equation}	
where the operators $\mathcal{H}$ and $\mathcal{L}$ are defined by 
\begin{equation}\label{eq:2.5-add}
\begin{aligned}
&\mathcal{H}(f)= -P_{>0}(e^{\frac{3s}{2}}\partial_X)\big((I-e^{3 s}\partial_X^2)^{\frac{\alpha-1}{2}} e^{\frac{3s}{2}}\partial_X f\big),\\
&\mathcal{L}(f)=-P_{\leq  0}(e^{\frac{3s}{2}}\partial_X)\big((I-e^{3 s}\partial_X^2)^{\frac{\alpha-1}{2}} e^{\frac{3s}{2}}\partial_X f\big).
\end{aligned}
\end{equation}

To ease notation, we set
\begin{equation}\label{def:1}
\begin{aligned}
\beta_{\tau}:=\frac{1}{1-\dot{\tau}},\
V=\frac{3}{2}X+\frac{1}{1-\dot{\tau}}\big(U+e^{\frac{s}{2}}(\kappa-\dot{\xi})\big),
\end{aligned}
\end{equation}
and differentiate \eqref{eq:2.5} $n$ times to find
	\begin{equation}\label{eq:2.6}
	\begin{aligned}
	&\bigg(\partial_{s}+\frac{3n-1}{2}+\beta_{\tau}(n+1_{n\neq1})\partial_{X}U\bigg)\partial_{X}^{n}U+V\partial_{X}^{n+1}U \\
	&=\beta_{\tau}e^{-s}\big(\mathcal{H}(\partial_{X}^{n}U)+\mathcal{L}(\partial_{X}^{n}U)\big)
-\beta_{\tau}1_{n\geq  3}\sum_{k=2}^{n-1}\binom{n}{k}\partial_{X}^{k}U\partial_{X}^{n+1-k}U. 
	\end{aligned}
	\end{equation}

\subsection{The stable blowup self-similar Burgers profile}

The inviscid Burgers has a unique stable self-similar solution (see for instance \cite{MR2470260}) 
     \begin{equation*}
	\begin{aligned}
	u(x,t)=(T_*-t)^{\frac{1}{2}}\overline{U}\bigg(\frac{x-x_*}{(T_*-t)^{\frac{3}{2}}}\bigg),
	\end{aligned}
	\end{equation*}
where $T_{*}$ and  $x_*$ are the blowup time and location, respectively, and $\overline{U}$ solves the steady profile equation for the Burgers problem
     \begin{equation}\label{eq:2.7}
	\begin{aligned}
	-\frac{1}{2}\overline{U}+\bigg(\frac{3}{2}X+\overline{U}\bigg)\partial_X\overline{U}=0.
	\end{aligned}
	\end{equation}
The above equation has an implicit family of solutions $\Psi_\nu$ satisfying 
     \begin{equation*}
	\begin{aligned}
	X=-\Psi_\nu-\frac{\nu}{6}\Psi_\nu^3,
	\end{aligned}
	\end{equation*}
where $\frac{\nu}{6}$ is the constant of integration, which can be solved via Cardano's formula yielding the explicit solution (ground state)
     \begin{equation*}
	\begin{aligned}
	\overline{U}=\left(-\frac{X}{2}+\bigg(\frac{1}{27}+\frac{X^2}{4}\bigg)^{\frac{1}{2}}\right)^{\frac{1}{3}}-\left(\frac{X}{2}+\bigg(\frac{1}{27}+\frac{X^2}{4}\bigg)^{\frac{1}{2}}\right)^{\frac{1}{3}},
	\end{aligned}
	\end{equation*}
where one has taken the convention $\overline{U}=\Psi_6$.
It is easy to check that
     \begin{equation}\label{eq:2.8}
	\begin{aligned}
&\overline{U}(0)=0, \quad \partial_X\overline{U}(0)=-1,
\quad \partial_X^2\overline{U}(0)=0,\\
&\partial_X^3\overline{U}(0)=6,\quad \partial_X^{2n}\overline{U}(0)=0,\ n=2,3,...,
	\end{aligned}
	\end{equation}
and
     \begin{equation}\label{eq:2.9}
	\begin{aligned}
|\partial_{X}^{i}\overline{U}(x)|\lesssim \langle x\rangle^{\frac{1}{3}-i},\quad i=0,1,2,...,5.
	\end{aligned}
	\end{equation}
In \eqref{eq:2.9}, the constant can be taken as  $1$ when  $i=0,1,2$. We also need the following  sharper bound 
     \begin{equation}\label{eq:2.9-add}
	\begin{aligned}
\frac{1}{4}\langle x\rangle^{-\frac{2}{3}} \leq |\partial_{X}\overline{U}(x)|\leq \frac{7}{20}\langle x\rangle^{-\frac{2}{3}}
	\end{aligned}
	\end{equation}
for $|X|\geq 100$.

One can also obtain Taylor expansion near $X=0$ and $X=\infty$ and 
other decay properties in different regions of $X$ for $\overline{U}$, we refer to \cite{MR4321245} for more details.

As a part of the main result, we will show that $U$ converges asymptotically to $\overline{U}$, so it is necessary to work with the equations on $\widetilde{U}:= U - \overline U$.

\subsection{Equations on $\partial_X^n\widetilde{U}$}

It follows from \eqref{eq:2.5} and \eqref{eq:2.7} that
	\begin{equation}\label{eq:3.0}
	\begin{aligned}
	\bigg(\partial_{s}-\frac{1}{2}+\beta_{\tau}\partial_{X}\overline{U}\bigg)\widetilde{U}+V\partial_{X}\widetilde{U}
	&=-\beta_{\tau}e^{-\frac{s}{2}}\dot{\kappa}+\beta_{\tau}e^{-s}\big(\mathcal{H}(U)+\mathcal{L}(U)\big)\\
&\quad-\beta_{\tau}\overline{U}^{\prime}\big(\dot{\tau}\overline{U}+e^{\frac{s}{2}}(\kappa-\dot{\xi})\big).
	\end{aligned}
	\end{equation}
Differentiating \eqref{eq:3.0} $n$ times yields 
	\begin{equation}\label{eq:3.1}
	\begin{aligned}
	&\bigg(\partial_{s}+1+\beta_{\tau}(\partial_{X}\tilde{U}+2\partial_{X}\bar{U})\bigg)\partial_{X}\widetilde{U}
	+V\partial_{X}^{2}\widetilde{U} \\
	&=\beta_{\tau}e^{-s}\big(\mathcal{H}(\partial_{X}U)+\mathcal{L}(\partial_{X}U)\big)-\beta_{\tau}F_{\widetilde{U}}^{(1)},\quad n=1,\\
&\bigg(\partial_{s}+\frac{3n-1}2+\beta_{\tau}(n+1)\partial_{X}U\bigg)\partial_{X}^{n}\widetilde{U}
	+V\partial_{X}^{n+1}\widetilde{U} \\
	&=\beta_{\tau}e^{-s}\big(\mathcal{H}(\partial_{X}^{n}U)+\mathcal{L}(\partial_{X}^{n}U)\big)-\beta_{\tau}F_{\widetilde{U}}^{(n)},\quad n\geq 2,
	\end{aligned}
	\end{equation}
where the forcing term has the following expression:
	\begin{equation}\label{eq:3.2}
	\begin{aligned}
F_{\widetilde{U}}^{(1)}&=
\partial_{X}^{2}\overline{U}\big(\dot{\tau}\overline{U}+e^{\frac{s}{2}}(\kappa-\dot{\xi})+\widetilde{U}\big)+\dot{\tau}(\partial_{X}\overline{U})^{2},\quad n=1,\\
	F_{\widetilde{U}}^{(n)}:&=\sum_{k=0}^{n}\binom{n}{k}\partial_{X}^{k+1}\overline{U}\big(\dot{\tau}\overline{U}^{(n-k)}+1_{k=n}e^{\frac{s}{2}}(\kappa-\dot{\xi})\big)  \\
	&\quad+\sum_{k=0}^{n-1}{\binom{n+1}{k}}\partial_{X}^{k}\widetilde{U}\partial_{X}^{n-k+1}\overline{U} \\
	&\quad+1_{n\geq  3}\sum_{k=1}^{n-2}\binom{n}{k}\partial_{X}^{k+1}\tilde{U}\partial_{X}^{n-k}\widetilde{U},\quad n\geq 2.
	\end{aligned}
	\end{equation}

\subsection{Constraints on $U$ and equations of $\xi, \tau$  and  $\kappa$}
	
To fix the dynamic modulation variables  $\xi, \tau$  and  $\kappa$, one shall impose some constraints on $U$. More precisely, we impose the following constraints at $X=0$ for $U$:
	\begin{equation}\label{eq:3.3}
     \begin{aligned}
	U(0,s)=0,\quad\partial_XU(0,s)=-1,\quad \partial_X^{2}U(0,s)=0. 
     \end{aligned}
	\end{equation}	
Then, one inserts \eqref{eq:3.3} into \eqref{eq:2.6} with $n=1$ and $n=2$ to get
           \begin{equation}\label{eq:3.4}
            \begin{aligned}
		\dot{\tau}=e^{-s}\big(\mathcal{H}(\partial_XU)+\mathcal{L}(\partial_XU)\big)(0,s),
            \end{aligned}
		\end{equation}	
and 
          \begin{equation}\label{eq:3.5}
            \begin{aligned}
		e^{\frac{s}{2}}(\kappa-\dot{\xi})=\frac{1}{\partial_{X}^{3}U(0,s)}
e^{-s}\big(\mathcal{H}(\partial_X^2U)+\mathcal{L}(\partial_X^2U)\big)(0,s), 	
             \end{aligned}
		\end{equation}
respectively. On the other hand, it follows from \eqref{eq:2.5}, \eqref{eq:3.3} and \eqref{eq:3.5} that
		\begin{equation}\label{eq:3.6}
            \begin{aligned}
		e^{-\frac{s}{2}}\dot{\kappa}=e^{\frac{s}{2}}(\kappa-\dot{\xi})+e^{-s}\big(\mathcal{H}(U)+\mathcal{L}(U)\big)(0,s).	
            \end{aligned}
		\end{equation}

	\subsection{Assumptions on the initial data}
Let $M\gg 1$ be a large number, and $\varepsilon=\varepsilon(M,\alpha)\ll 1$ be a small positive number, both of which will be determined later in the proof. For convenience, we also introduce another two parameters 
$$s_0=-\log\varepsilon,\quad h=(\log M)^{-1},$$
and denote
$$U_{0}(\cdot)=U(\cdot, s_0), \quad \widetilde{U}_{0}(\cdot)=\widetilde{U}(\cdot, s_0).$$
Finally, let $\delta$ be a small positive number.

In the following, we select the initial data used in  the main result. It suffices to give the initial data in self-similar variables, which can be easily translated into the corresponding ones in the physical variables using \eqref{eq:2.1}-\eqref{eq:2.3}.

First, to meet the constraints \eqref{eq:3.3}, one shall impose the following conditions at $X=0$:
	\begin{equation}\label{eq:3.7}
      \begin{aligned}
	U_{0}(0)=0,\quad\partial_{X}U_{0}(0)=\operatorname*{min}_{X}\partial_{X}U_{0}=-1,\quad\partial_{X}^{2}U_{0}(0)=0.
     \end{aligned}
	\end{equation}
Next, we impose the following conditions in different regions of $X$:\\
		\begin{equation}\label{eq:3.8}
          \left\{
           \begin{aligned}
		&|\widetilde{U}_0(X)|\leq  \frac{1}{2}\varepsilon^{\frac{1}{3}} h^4, &&\quad 0\leq  |X|\leq   h,\\
		&|\widetilde{U}_0(X)|\leq  \varepsilon^{\frac{1}{2}}\langle X\rangle^{\frac{1}{3}}, &&\quad h\leq  |X|\leq  \frac{1}{2}\varepsilon^{-\frac{3}{2}},
		\end{aligned}
           \right.
		\end{equation}
        \begin{equation}\label{eq:3.9}
             \left\{
           \begin{aligned}
          &|\partial_X\widetilde{U}_0(X)|\leq  \frac{1}{2}\varepsilon^{\frac{1}{3}} h^3, &&\quad 0\leq  |X|\leq   h,\\
		&|\partial_X\widetilde{U}_0(X)|\leq  \varepsilon^{\frac{1}{4}}\langle X\rangle^{-\frac{2}{3}},&&\quad h\leq  |X|\leq  \frac{1}{2}\varepsilon^{-\frac{3}{2}},\\
		&|\partial_{X}U_0(X)|\leq \varepsilon^{-1}, &&\quad \frac{1}{2}\varepsilon^{-\frac{3}{2}}\leq  |X|<\infty,
		\end{aligned}
            \right.
		\end{equation}
           \begin{equation}\label{eq:4.0}
            \left\{
           \begin{aligned}
		&|\partial_{X}^{2}\widetilde{U}_0(X)|\leq  \frac{1}{2}\varepsilon^{\frac{1}{3}}  h^{2}, &&\quad 0\leq  |X|\leq   h,\\
		&|\partial_{X}^{2}U_0(X)|\leq  M^{\frac15 -\delta}, &&\quad h\leq  |X|<\infty,
		\end{aligned}
           \right.
		\end{equation}
and 
		\begin{equation}\label{eq:4.1}
             \left\{
          \begin{aligned}
		&|\partial_X^3\widetilde{U}_0(X)|\leq  \frac{1}{4}\varepsilon^{\frac{1}{3}} h, &&\quad 0\leq  |X|\leq   h,\\
		&|\partial_{X}^{4}\widetilde{U}_0(X)|\leq  \frac{1}{4}\varepsilon^{\frac{1}{3}}, &&\quad 0\leq  |X|\leq   h.
           \end{aligned}      
           \right.
		\end{equation}
Finally, we impose the following global $L^2$ and $L^\infty$ constraints:
     \begin{equation}\label{eq:4.2}
             \left\{
          \begin{aligned}
		&\|\partial_XU_0\|_{L^2}\leq  50,\\
           &\|\partial_X^5U_0\|_{L^2}\leq  \frac{1}{2}M^{\frac{3}{2}},\\
           &\|\partial_X^4U_0\|_{L^\infty}\leq  \frac{1}{2} M,\\
           &\|U_0+\varepsilon^{-\frac{1}{2}}\kappa_0\|_{L^{\infty}} \leq  \frac{1}{2}M\varepsilon^{-\frac{1}{2}},
           \end{aligned}      
           \right.
	\end{equation}
and the precise third derivative bootstrap at $X=0$:
	\begin{equation}\label{eq:4.2-add}
	|\partial_X^3\widetilde U_0(0)|\leq  \frac{1}{4}\varepsilon^{\frac{1}{2}}.
	\end{equation}

Besides, we also  assume 
           \begin{equation}\label{eq:4.3}
            \begin{aligned}
	\supp_{x} {u_0}\subset [-1,1],\quad |\kappa_0|\leq M.
             \end{aligned}
		\end{equation}

	\subsection{Main results}
The main result in this paper is stated as follows:
	\begin{theorem}~\label{thm2.1} Let $\alpha<0$.
There exist a sufficiently large $M>0$ and a sufficiently small $\varepsilon=\varepsilon(M,\alpha)>0$ such that 
if the initial data $u_{0}$ satisfies \eqref{eq:3.7}-\eqref{eq:4.3}, then there exists a unique solution
$u\in C([-\varepsilon, T_*), H^5(\mathbb{R}))$ to the Cauchy problem of \eqref{eq:1.1} such that the following hold:\\                
(I) the blowup time and location satisfy 
$$T_{*}\leq 2\varepsilon^{\frac{7}{4}}\ \text{and}\  |x_{*}| \leq 3M \varepsilon,$$
 respectively. \\
(II) $u$ is bounded: $\|u(\cdot, t)\|_{L^\infty}\leq M$ for all $t\in [-\varepsilon, T_*]$.\\ 
(III) $\partial_{x}u$ blows up and has the following blowup rate
	       \begin{equation*}
             \begin{aligned}
			\frac{1}{3(T_{*}-t )}\leq \|\partial_{x}u\|_{L^{\infty}} \leq \frac{3}{T_{*}-t}.
		\end{aligned}
		\end{equation*}
(IV) $u$ displays a cusp singularity at $(x_{*}, T_{*})$ with $u(\cdot, T_{*})\in C^{\frac{1}{3}}(\mathbb{R})$.\\
(V) $u$ converges asymptotically in the self-similar variables to a stable self-similar solution 
$\overline{U}_{\nu}$ of the inviscid Burgers equation, namely
			 \begin{equation}\label{eq:4.5}
                \begin{aligned}
			\limsup _{s \rightarrow \infty} \|U-\overline{U}_{\nu} \|_{L^{\infty}}=0,
		    \end{aligned}
		    \end{equation}
where $\overline{U}_{\nu}$ is defined by
     \begin{equation}\label{eq:4.6}
     \begin{aligned}
	\overline{U}_{\nu}(X)=\left(\frac{\nu}{6}\right)^{-\frac{1}{2}}\overline{U}\left(\left(\frac{\nu}{6}\right)^{\frac{1}{2}}X\right)
      \end{aligned}
	\end{equation}
with  $\nu=\lim _{s \rightarrow \infty} \partial_{X}^{3} U(0,s)$.

	\end{theorem}

	The blowup results in Theorem \ref{thm2.1} are stable under perturbation, so one can relax the initial conditions above to achieve the following:
	
	\begin{corollary}~\label{coro2.1.1}
		There exists an open set of initial data in the  $H^{5}$  topology satisfying the hypothesis in Theorem \ref{thm2.1} such that the conclusions of Theorem \ref{thm2.1} still hold.
	\end{corollary}

The proof of Corollary \ref{coro2.1.1} is quite similar to  \cite{MR4321245}, so we will omit it. 
In the remaining sections, we will focus on the preliminaries and the proof of Theorem \ref{thm2.1}.

	\section{Bootstrap assumptions}\label{bootstraps}
	
	\subsection{Bootstrap assumptions for the self-similar variables}~\label{sec3.1}
First, we make the following bootstrap assumptions for various derivatives in different regimes:
         \begin{equation}\label{eq:5.0}
          \left\{
           \begin{aligned}
		&|\widetilde{U}(X,s)|\leq  \varepsilon^{\frac{1}{3}} h^{4}, &&\quad 0\leq  |X|\leq   h,\\
		&|\widetilde{U}(X,s)|\leq  \varepsilon^{\frac{1}{4}}\langle X\rangle^{\frac{1}{3}}, &&\quad h\leq  |X|\leq \frac{1}{2}e^{\frac{3s}{2}},
		\end{aligned}
           \right.
		\end{equation}
       \begin{equation}\label{eq:5.1}
          \left\{
           \begin{aligned}
		&|\partial_{X}\widetilde{U}( X,s)| \leq  \varepsilon^{\frac{1}{3}} h^{3}, &&\quad 0\leq  |X|\leq   h, \\
		&|\partial_{X}\widetilde U(X,s)| \leq  \varepsilon^{\frac{1}{8}}\langle X\rangle^{-\frac{2}{3}}, &&\quad h\leq  |X|\leq   \frac{1}{2}e^{\frac{3s}{2}},\\
          &|\partial_{X} U(X,s)| \leq   2e^{-s}, &&\quad  \frac{1}{2}e^{\frac{3s}{2}}\leq  |X|<\infty,
		\end{aligned}
           \right.
		\end{equation}
        \begin{equation}\label{eq:5.2}
          \left\{
           \begin{aligned}
		&|\partial_X^2\widetilde{U}(X,s)|\leq  \varepsilon^{\frac{1}{3}} h^2, &&\quad 0\leq  |X|\leq   h, \\
		&|\partial_X^2U( X,s)|\leq   M^{\frac{1}{5}}, &&\quad h\leq  |X|<\infty,
		\end{aligned}
           \right.
		\end{equation}
and
          \begin{equation}\label{eq:5.3}
          \left\{
           \begin{aligned}
		&|\partial_X^3\widetilde U(X,s)|\leq  \varepsilon^{\frac{1}{3}} h, &&\quad 0\leq  |X|\leq   h, \\
		&|\partial_X^4\widetilde U(X,s)|\leq  \varepsilon^{\frac{1}{3}}, &&\quad 0\leq  |X|\leq   h.
		\end{aligned}
           \right.
		\end{equation}
Next, we assume the global $L^\infty$ bounds	
     \begin{equation}\label{eq:5.4}
     \left\{
      \begin{aligned}
	&\|\partial_X^4U\|_{L^\infty}\leq M,\\
     &\|U+e^{\frac{s}{2}}\kappa\|_{L^{\infty}} \leq  Me^{\frac{s}{2}},
     \end{aligned}
           \right.
	\end{equation}
and the precise third derivative bootstrap at $X=0$:
	\begin{equation}\label{eq:5.5}
	|\partial_X^3\widetilde U(0,s)|\leq  \varepsilon^{\frac{1}{2}}.
	\end{equation}

\subsection{Bootstrap assumptions for the dynamic modulations variables}
We assume that the dynamic modulation variables satisfy
          \begin{equation}\label{eq:5.6}
          \begin{aligned}
		|\dot{\tau}(t)| \leq e^{-\frac{3s}{4}},\quad |\tau(t)|\leq  2\varepsilon^\frac{7}{4},\quad T_*\leq 2\varepsilon^\frac{7}{4},
          \end{aligned}
		\end{equation}
and
         \begin{equation}\label{eq:5.7}
          \begin{aligned}
		|\dot{\xi}(t)| \leq 2 M,\quad |\xi(t)|\leq  3M\varepsilon.
          \end{aligned}
		\end{equation}

\subsection{Consequences of the bootstrap assumptions}~\label{sec3.2}
By the bootstrap assumptions in Subsection \ref{sec3.1}, one can get some direct consequences which will be frequently used later.

(I) The bound on $\|\partial_XU\|_{L^\infty}$. It holds that
		\begin{equation}\label{eq:5.8}
		\frac{99}{100}\leq \|\partial_XU\|_{L^\infty}\leq  \frac{101}{100}. 
		\end{equation}
We only show the second inequality in \eqref{eq:5.8} since the other one can be proven similarly. When $0\leq |X| \leq h$, by \eqref{eq:2.9} and \(\eqref{eq:5.1}_1\), one has
		\begin{equation*}
           \begin{aligned}
		|\partial_{X} U(X, s)|& \leq |\partial_{X} \widetilde{U}(X, s)|+|\partial_{X} \overline{U}(X)| \\
&\leq \varepsilon^{\frac{1}{3}} h^3+1 \leq \frac{101}{100}.
           \end{aligned}
		\end{equation*}
When $h \leq |X| \leq \frac{1}{2}e^{\frac{3s}{2}}$, it follows from  
\eqref{eq:2.9} and \(\eqref{eq:5.1}_2\) that
          \begin{equation*}
		\begin{aligned}
		|\partial_{X} U(X, s)| & \leq |\partial_{X} \widetilde{U}(X, s)|+|\partial_{X} \overline{U}(X)| \\
		& \leq (\varepsilon^{\frac{1}{8}}+1)\langle h\rangle^{-\frac{2}{3}}\leq \frac{101}{100}.
		\end{aligned}
		\end{equation*}
When $\frac{1}{2}e^{\frac{3s}{2}}\leq  |X|<\infty$, one can use \(\eqref{eq:5.1}_3\) to obtain
          \begin{equation*}
		\begin{aligned}
		|\partial_{X} U(X, s)| \leq  2e^{-s}\leq  2\varepsilon \leq \frac{101}{100}
		\end{aligned}
		\end{equation*}
due to $s\geq s_0$. 

We comment that one can furthermore show by using more delicate decay of $\overline{U}$ that 
           \begin{equation}\label{extremum}
		\begin{aligned}
		\|\partial_{X} U(X,s)\|_{L^\infty}=1=-\partial_{X}U(0,s)
		\end{aligned}
		\end{equation}
with extremum attained uniquely at $X=0$,  see \cite{MR4321245}.

(II) The bounds on $\|\partial_X^2U\|_{L^\infty}$ and $\|\partial_X^3U\|_{L^\infty}$.
By \(\eqref{eq:5.2}_1\), one first notice that	for $0\leq |X| \leq h$
          \begin{equation*}
           \begin{aligned}
		|\partial_{X}^{2} U(X,s)|\leq |\partial_{X}^2 \widetilde{U}(X,s)|+|\partial_{X}^2 \overline{U}(X)|
\leq \varepsilon^{\frac{1}{3}} h^2+\langle h\rangle^{-\frac{5}{3}}\leq  M^{\frac{1}{5}}. 
           \end{aligned}
		\end{equation*}
This together with \(\eqref{eq:5.2}_2\) yields
		\begin{equation}\label{eq:5.9}
            \begin{aligned}
		\|\partial_{X}^{2} U\|_{L^{\infty}} \leq  M^{\frac{1}{5}}, 
            \end{aligned}
		\end{equation}
which together with \eqref{eq:5.4} by using the Gagliardo-Nirenberg interpolation inequality yields 
		\begin{equation}\label{eq:6.0}
		\|\partial_{X}^{3} U\|_{L^{\infty}} \lesssim \|\partial^{2}_{X} U\|^{\frac{1}{2}}_{L^{\infty}}  \|\partial_{X}^{4} U\|^{\frac{1}{2}}_{L^{\infty}} \lesssim M^{\frac{3}{5}},
		\end{equation}
where the implicit constant is universal.

(III) The bound on $\beta_{\tau}$. 
From \eqref{def:1} and \eqref{eq:5.6}, it follows that
		\begin{equation}\label{eq:6.1}
		\begin{aligned}
		\frac{99}{100} \leq \frac{1}{1+4\varepsilon^\frac{7}{4}} \leq \beta_{\tau}=\frac{1}{1-\dot{\tau}} \leq   \frac{1}{1-4\varepsilon^\frac{7}{4}} \leq   \frac{101}{100}.
		\end{aligned}
		\end{equation}

(IV) The bound on $|\partial_X^3U(0,s)|$. One may use \eqref{eq:2.8} and \eqref{eq:5.5} to estimate 
           \begin{equation*}
		\begin{aligned}
		|\partial_X^3U(0,s)|\leq |\partial_X^3\widetilde U(0,s)|+|\partial_X^3\overline{U}(0,s)|
\leq \varepsilon^{\frac{1}{2}}+6,
		\end{aligned}
		\end{equation*}
which implies 
           \begin{equation}\label{eq:6.3}
		\begin{aligned}
		5\leq \partial_X^3U(0,s)\leq 7.
		\end{aligned}
		\end{equation}

\section{$L^{2}$  estimates}\label{l2 bounds}

	\subsection{Low-order  $L^{2}$  bounds}

\begin{lemma}
		It holds that
		\begin{equation}\label{eq:7}
           \begin{aligned}
		\|U+e^{\frac{s}{2}}\kappa\|_{L^{2}} \leq   Me^{\frac{5s}{4}}. 
           \end{aligned}
		\end{equation}
\end{lemma}
\begin{proof} 
It follows from \eqref{eq:2.4} that
\begin{equation*}
\begin{aligned}
	\|U+e^{\frac{s}{2}}\kappa\|_{L^2}=e^{\frac{5s}{4}}\|u\|_{L^2}
=e^{\frac{5s}{4}}\|u_0\|_{L^2} \leq M	e^{\frac{5s}{4}},
\end{aligned}
\end{equation*}
where one has used the fact that \eqref{eq:1.1} is $L^2$ conserved and $\|u_0\|_{L^2}\leq M$ by \(\eqref{eq:4.2}_4\) and \eqref{eq:4.3}. 
\end{proof}

	\begin{lemma} It holds that
		\begin{equation}\label{eq:7.1}
           \begin{aligned}
		\|\partial_{X} U\|_{L^{2}} \leq   100.
           \end{aligned}
		\end{equation}
	\end{lemma}
	\begin{proof} Taking $L^2$ inner product of \eqref{eq:2.6} for $n=1$ with $\partial_{X}U$ and integrating over $X$ yields
		\begin{equation}\label{eq:7.2}
		\begin{aligned}
\frac{1}{2} \frac{d}{d s}\|\partial_{X} U\|_{L^{2}}^{2}+\|\partial_{X} U\|_{L^{2}}^{2}
+\int V \partial_{X} U \partial_{X}^{2} U\, d X
+\beta_{\tau} \int (\partial_{X} U)^{3}\, d X=0.		
		\end{aligned}
		\end{equation}
By integration by parts, one obtains
		\begin{equation}\label{eq:7.3}
		\begin{aligned}
		\int V \partial_{X} U \partial_{X}^{2} U\, d X=-\frac{\beta_{\tau}}{2} \int (\partial_{X} U)^{3} d X-\frac{3}{4}\|\partial_{X} U\|_{L^{2}}^{2}.
		\end{aligned}  
		\end{equation}
Substituting \eqref{eq:7.3} into \eqref{eq:7.2} gives  
           \begin{equation}\label{eq:7.5}
		\begin{aligned}
\frac{1}{2} \frac{d}{d s}\|\partial_{X} U\|_{L^{2}}^{2}+\frac{1}{4}\|\partial_{X} U\|_{L^{2}}^{2}
=\underbrace{-\frac{\beta_{\tau}}{2} \int (\partial_{X} U)^{3}\, d X}_{I}. 		
		\end{aligned}
		\end{equation}

By \eqref{eq:5.1}, one may estimate 
		\begin{equation}\label{eq:7.6}
		\begin{aligned}
		|I|&\leq \left(\int_{0 \leq  |X| \leq   h}+\int_{h \leq  |X| \leq   \frac{1}{2}e^{\frac{3s}{2}}}+\int_{\frac{1}{2}e^{\frac{3s}{2}}\leq  |X|<\infty}\right)\left(\partial_{X} U\right)^{3}\, d X \\
		&\leq    \sup_{0 \leq  |X| \leq   h}(|\partial_{X} \widetilde{U}(X, s)|+|\partial_{X} \overline{U}(X)|)^3 h\\
&\quad+\int_{h \leq   |X| \leq   \frac{1}{2}e^{\frac{3s}{2}}} (|\partial_{X} \widetilde{U}(X, s)|+|\partial_{X} \overline{U}(X)|)^3\, d X \\
		&\quad+2e^{-s}\int_{\frac{1}{2}e^{\frac{3s}{2}} \leq  |X|<\infty} |\partial_{X} U|^{2}\, d X \\
          &\leq   (\varepsilon^{\frac{1}{3}} h+1)^3 h
+(\varepsilon^{\frac{1}{8}}+1) \int \frac{1}{1+X^{2}}\, d X 
		+2e^{-s} \int |\partial_{X} U|^{2}\, d X \\
		&\leq  8 h+\pi+2\varepsilon\|\partial_{X} U\|_{L^{2}}^{2}.
		\end{aligned}
		\end{equation}
It follows from \eqref{eq:7.5} and \eqref{eq:7.6} that
		\begin{equation*}
		\begin{aligned}
		\frac{1}{2}\frac{d}{ds}\|\partial_{X}U\|_{L^{2}}^{2}+\bigg(\frac{1}{4}-2\varepsilon\bigg)\|\partial_{X}U\|_{L^{2}}^{2}\leq 8 h+\pi,
		\end{aligned}
		\end{equation*}
namely
           \begin{equation}\label{eq:7.7}
		\begin{aligned}
		\frac{d}{ds}\|\partial_{X}U\|_{L^{2}}^{2}+\frac{1}{4}\|\partial_{X}U\|_{L^{2}}^{2}
\leq 16 h+2\pi.
		\end{aligned}
		\end{equation}
Solving \eqref{eq:7.7} by Gronwall's inequality along with \(\eqref{eq:4.2}_1\) produces
		\begin{equation*}
		\begin{aligned}
		\|\partial_{X} U \|_{L^{2}}^{2}&\leq   \|\partial_{X} U_0\|_{L^{2}}^{2}e^{-\frac{1}{4}(s-s_{0})}+(16 h+2\pi)\int^{s}_{s_{0}}e^{-\frac{1}{4}(s-s^{\prime})}\, d s^{\prime}\\
		&\leq 50^{2}e^{-\frac{1}{4}(s-s_{0})}
+4(16h+2\pi)\big(1-e^{-\frac{1}{4}(s-s_{0})}\big)\leq 100^2.
		\end{aligned}
		\end{equation*}
This completes the proof of \eqref{eq:7.1}.

	\end{proof}

	\subsection{Top-order $L^2$ bound.} 
	\begin{lemma}
		It holds that
		\begin{equation}\label{eq:8.0}
           \begin{aligned}
		\|\partial_{X}^{5}U\|_{L^{2}}\leq M^{\frac{3}{2}}.
           \end{aligned}
		\end{equation}
	\end{lemma}
	
\begin{proof}

\eqref{eq:2.6} with $n=5$ reads as follows:
	\begin{equation*}
	\begin{aligned}
	&\left(\partial_{s}+7+6\beta_{\tau}\partial_{X}U\right)\partial_{X}^5U+V\partial_{X}^6U \\
	&=\beta_{\tau}e^{\frac{s}{2}}(I-e^{3 s}\partial_X^2)^{\frac{\alpha-1}{2}} \partial_X^6U-\beta_{\tau}\sum_{k=2}^4\binom{5}{k}\partial_{X}^{k}U\partial_{X}^{6-k}U,
	\end{aligned}
	\end{equation*}
which multiplying by $\partial_X^5 U$ and integrating over $X$ imply
		\begin{equation}\label{eq:8.1}
		\begin{aligned}
		&\frac{1}{2}\frac{d}{ds}\|\partial_{X}^{5}U\|_{L^{2}}^{2}+7\|\partial_{X}^{5}U\|_{L^{2}}^{2}+6\beta_{\tau}\int \partial_{X}U(\partial_{X}^{5}U)^{2}\:dX \\
		&+\int V\partial_{X}^{5}U\partial_{X}^{6}U\, dX 
		=-\beta_{\tau}\sum_{k=2}^{4}\binom{5}{k}\int \partial_{X}^{5}U\partial_{X}^{k}U\partial_{X}^{6-k}U\:dX.
		\end{aligned}  
		\end{equation}
By integration by parts, one obtains
		\begin{equation}\label{eq:8.2}
		\begin{aligned}
		\int V\partial_{X}^{5}U\partial_{X}^{6}U\, dX=-\frac{\beta_{\tau}}{2} \int \partial_{X}U(\partial_{X}^5 U)^{2} d X-\frac{3}{4}\|\partial_{X}^5 U\|_{L^{2}}^{2}.
		\end{aligned}  
		\end{equation}
It follows from \eqref{eq:8.1} and \eqref{eq:8.2} that
		\begin{equation}\label{eq:8.3}
		\begin{aligned}
&\frac{1}{2}\frac{d}{ds}\|\partial_X^5U\|_{L^2}^2+\frac{25}{4}\|\partial_X^5U\|_{L^2}^2+\underbrace{\frac{11}{2}\beta_\tau\int \partial_{X}U(\partial_X^5U)^2\, dX}_{I}\\
&= \underbrace{-15\beta_\tau\int \partial_X^2U\partial_X^4U\partial_X^{5}U\,dX}_{II}\underbrace{-10\beta_\tau\int (\partial_X^3U)^2\partial_X^{5}U\,dX.}_{III} 
           \end{aligned}
		\end{equation}

By \eqref{eq:5.8} and \eqref{eq:6.1}, one may estimate
		\begin{equation*}
		|I|\leq  {\frac{11}{2}}\times\frac{101}{100}\times\frac{101}{100}\|\partial_{X}^{5}U\|_{L^{2}}^{2}\leq\bigg({\frac{11}{2}}+{\frac{1}{4}}\bigg)\|\partial_{X}^{5}U\|_{L^{2}}^{2}.
		\end{equation*}
To handle the RHS of \eqref{eq:8.3}, we first use the Gagliardo-Nirenberg interpolation inequality to deduce 	
			\begin{equation}\label{eq:8.5}
                \begin{aligned}
			&\|\partial_{X}^{3}U\|_{L^{2}} \lesssim \|\partial_{X}U\|_{L^{\infty}}^{\frac{4}{7}}\|\partial_{X}^{5}U\|_{L^{2}}^{\frac{3}{7}}\leq 
\|\partial_{X}^{5}U\|_{L^{2}}^{\frac{3}{7}},\\
			&\|\partial_{X}^{4}U\|_{L^{2}} \lesssim \|\partial_{X}U\|_{L^{\infty}}^{\frac{2}{7}}\|\partial_{X}^{5}U\|_{L^{2}}^{\frac{5}{7}}\leq 
\|\partial_{X}^{5}U\|_{L^{2}}^{\frac{5}{7}},
                 \end{aligned}
			\end{equation}
and then use \eqref{eq:8.5} to estimate
		\begin{equation}\label{eq:8.6}
		\begin{aligned}
		|II|& \lesssim\|\partial_{X}^{2}U\|_{L^{\infty}}\|\partial_{X}^{4}U\|_{L^{2}}\|\partial_{X}^{5}U\|_{L^{2}}\\
		&\lesssim M^{\frac{1}{5}}\|\partial_{X}^{5}U\|_{L^{2}}^{\frac{12}{7}}
		\leq   CM^{\frac{21}{10}}+\frac{1}{16}\left\|\partial_{X}^{5}U\right\|_{L^{2}}^{2},
		\end{aligned}
		\end{equation}
and
          \begin{equation}\label{eq:8.7}
		\begin{aligned}
		|III|& \lesssim \|\partial_{X}^{3}U\|_{L^{\infty}}\|\partial_{X}^{3}U\|_{L^{2}}\|\partial_{X}^{5}U\|_{L^{2}}  \\
		&\lesssim M^{\frac{3}{5}}\|\partial_{X}^{5}U\|_{L^{2}}^{\frac{10}{7}}
		\leq   CM^{\frac{21}{10}}+\frac{1}{16}\left\|\partial_{X}^{5}U\right\|_{L^{2}}^{2}.
		\end{aligned}
		\end{equation}

It follows from \eqref{eq:8.3}, \eqref{eq:8.6} and \eqref{eq:8.7} that
		\begin{equation*}
		{\frac{1}{2}}{\frac{d}{ds}}\|\partial_{X}^{5}U\|_{L^{2}}^{2}+{\frac{1}{8}}\|\partial_{X}^{5}U\|_{L^{2}}^{2}\leq   CM^{\frac{21}{10}},  
		\end{equation*}
which together with Gronwall's inequality implies
\begin{equation*}\
\begin{aligned}
	\|\partial_{X}^{5}U\|_{L^{2}}^{2}&\leq \|\partial_{X}^{5}U_0\|_{L^{2}}^{2}e^{-4(s-s_0)}+CM^{\frac{21}{10}}e^{-4s}\int_{s_0}^s e^{4s^\prime}\, ds^\prime\\
&\leq \frac{1}{2} M^3+CM^{\frac{21}{10}}\big(1-e^{-4(s-s_0)}\big)\leq M^3,
\end{aligned}  
\end{equation*}
where one has used \(\eqref{eq:4.2}_2\). 
		
	\end{proof}

\subsection{Intermediate $L^2$ bounds.}
One may use the Gagliardo-Nirenberg interpolation inequality to get
\begin{equation}\label{eq:8.8}
\begin{aligned}
\|\partial_{X}^{2} U\|_{L^2}\lesssim \|\partial_{X} U\|_{L^{\infty}}^{\frac{5}{6}}\|\partial_{X}^4 U\|_{L^{\infty}}^{\frac{1}{6}}\lesssim M^{\frac{1}{6}},\\
\|\partial_{X}^{3} U\|_{L^2}\lesssim \|\partial_{X}^2 U\|_{L^{\infty}}^{\frac{3}{4}}\|\partial_{X}^4 U\|_{L^{\infty}}^{\frac{1}{4}}\lesssim M^{\frac{1}{4}},\\
\|\partial_{X}^{4} U\|_{L^2}\lesssim \|\partial_{X}^2 U\|_{L^{\infty}}^{\frac{1}{4}}\|\partial_{X}^4 U\|_{L^{\infty}}^{\frac{3}{4}}\lesssim M^{\frac{3}{4}}.
\end{aligned}
\end{equation}

\section{Lemmas on Fourier multiplier}\label{LFM}
	In this section, we study some analytic properties of the operators $\mathcal{H}$ and $\mathcal{L}$ defined in \eqref{eq:2.5-add}. We first consider the operator $\mathcal{L}$. Since $\mathcal{L}$ only contains low frequency of $e^{\frac{3s}{2}}\xi_X$, this helps to  transfer the derivatives into temporal decay by Bernstein's inequality.   
		
\begin{lemma}\label{FM-1}
We have 
\begin{equation*}
\begin{aligned}
		\| \mathcal{L}(\partial_X^lf)\|_{L^\infty} \lesssim e^{-(\frac{1}{2}+l)\frac{3s}{2}}\|f\|_{L^2}\qquad \mathrm{for}\ l \geq 0.
\end{aligned}
\end{equation*}
As a consequence, 
\begin{equation}\label{DL}
\begin{aligned}
\| \mathcal{L}(\partial_X ^{l}U)\|_{L^\infty} \leq C(M) e^{-\frac{3l-1}{2}s} \qquad \mathrm{for}\ l \geq 0.
\end{aligned}
\end{equation}
		
\end{lemma}

\begin{proof}
	Observe that
\begin{equation*}
\begin{aligned}
\|\mathcal{L}(\partial_X ^{l}f)\|_{L^2}
&=e^{-\frac{3s}{2}l}\|(e^{\frac{3s}{2}}\xi_X)^{l+1}\varphi_0(e^{\frac{3s}{2}}\xi_X)\big((1+e^{3 s}|\xi_X|^2)^{\frac{\alpha-1}{2}}\hat{f}(\xi_X)\big)\|_{L_{\xi_X}^2}\\
&\lesssim e^{-\frac{3s}{2}l}\|\hat{f}\|_{L_{\xi_X}^2}=e^{-\frac{3s}{2}l}\|f\|_{L^2},
\end{aligned}
\end{equation*}
where one has used the fact that $\varphi_0$ is supported on $\{e^{\frac{3s}{2}}|\xi_X|\leq 2\}$. 
This together with Bernstein's inequality yields 
\begin{equation*}
\begin{aligned}
&\|\mathcal{L}(\partial_X ^{l}f)\|_{L^\infty}\leq e^{-\frac{3s}{2}\times \frac{1}{2}}\|\mathcal{L}(\partial_X ^{l}f)\|_{L^2}\leq e^{-(\frac{1}{2}+l)\frac{3s}{2}}\|f\|_{L^2}.
\end{aligned}
\end{equation*}
Applying the above inequality to $f=U+e^{\frac{s}{2}}\kappa$ yields	
\begin{equation*}
\begin{aligned}
\|\mathcal{L}(\partial_X ^{l}U)\|_{L^\infty}
&=\|\mathcal{L}\big(\partial_X ^{l}(U+e^{\frac{s}{2}}\kappa)\big)\|_{L^\infty}
\leq e^{-(\frac{1}{2}+l)\frac{3s}{2}}\|U+e^{\frac{s}{2}}\kappa\|_{L^2}\\
&\leq Me^{-(\frac{1}{2}+l)\frac{3s}{2}}e^{\frac{5s}{4}}\leq C(M) e^{-\frac{3l-1}{2}s}.
\end{aligned}
\end{equation*}

\end{proof}

To estimate the operator $\mathcal{H}$, we write it in the following integral form. 

\begin{lemma}\label{FM-2}  For each $s$, there exists a function $G_s\in C^\infty(\mathbb{R}\setminus \{0\})$ such that
\begin{equation}\label{integral form}
\begin{aligned}
\mathcal{H}(f)(X)=e^{-\frac{3s}{2}}\int \mathrm{sign}(X-Y)G_{\alpha}\big(e^{-\frac{3s}{2}}(X-Y)\big)f(Y)\, d Y,
\end{aligned}
\end{equation}
where $G_{\alpha}$ is an even kernel satisfying
\begin{equation}\label{kernel}
\begin{aligned}
				&|G_{\alpha}(X)| \leq C\left\{\begin{array}{ll}
					\frac{1}{|X|^{1+\alpha}} &\quad \mathrm{for}\  |X|\leq 1\ \mathrm{and}\ \ -1<\alpha<0, \\
					\log \frac{1}{|X|}+1 &\quad \mathrm{for}\  |X|\leq 1\ \mathrm{and}\ \ \alpha=-1,\\
					1 &  \quad \mathrm{for}\  |X|\leq 1\ \mathrm{and}\ \ \alpha<-1,\\
					e^{-\frac{|X|}{2}}&  \quad \mathrm{for}\  |X|> 1\ \mathrm{and}\ \ \alpha<0.
				\end{array}\right.
\end{aligned}
\end{equation}

\end{lemma}

\begin{proof} The proof is a straightforward calculation by using the Fourier inversion formula together with the property of the Bessel potential (see \cite{MR4766878}).

\end{proof}

\begin{lemma}\label{FM-3}
It holds that
\begin{equation*}
\begin{aligned}
\|\mathcal{H}(f)\|_{L^\infty}\lesssim \|f\|_{L^\infty}.
\end{aligned}
\end{equation*}
As a consequence, 
\begin{equation*}
\begin{aligned}
\|\mathcal{H}(U)\|_{L^\infty} \leq C(M)e^{\frac{s}{2}},
\end{aligned}
\end{equation*}
and
\begin{equation*}
\begin{aligned}
\|\mathcal{H}(\partial_X ^{l}U)\|_{L^\infty} \leq C(M) \qquad \mathrm{for}\  1 \leq l \leq 4. 
\end{aligned}
\end{equation*}

\end{lemma}

\begin{proof}
Noticing from \eqref{kernel} that $G_{\alpha}\in L^1$, one may use Young's inequality to estimate	
\begin{equation*}
\begin{aligned}
\|\mathcal{H}(f)\|_{L^\infty}\leq e^{-\frac{3s}{2}}\|G_{\alpha}(e^{-\frac{3s}{2}}\cdot)\|_{L^1}\|f\|_{L^\infty}\lesssim \|f\|_{L^\infty}.
\end{aligned}
\end{equation*}

It follows from \(\eqref{eq:5.4}_2\) that
\begin{align*}
\|\mathcal{H}(U)\|_{L^\infty}=\|\mathcal{H}(U+e^{\frac{s}{2}}\kappa)\|_{L^\infty}\leq \|U+e^{\frac{s}{2}}\kappa\|_{L^\infty}\leq Me^{\frac{s}{2}},
\end{align*}
and from \eqref{eq:7.1} and \eqref{eq:8.0} that
\begin{align*}
\|\mathcal{H}(\partial_X ^{l}U)\|_{L^\infty}\leq \|\partial_X^{l}U\|_{L^\infty}\leq \|\partial_XU\|_{L^2}^{\frac{9-2l}{8}}\|\partial_X^5U\|_{L^2}^{\frac{2l-1}{8}}\leq C(M)
\end{align*}
for $1 \leq l \leq 4$.

\end{proof}
	
Collecting the results in Lemma \ref{FM-1} and \ref{FM-3} yields the following $L^\infty $bounds for $\mathcal{H}$ and $\mathcal{L}$. 

\begin{corollary}\label{FM-4} We have
\begin{equation}\label{HL-bound-1}
\begin{aligned}
\|\mathcal{H}(U)+\mathcal{L}(U)\|_{L^\infty} \leq C(M)e^{\frac{s}{2}},
\end{aligned}
\end{equation}
and 
\begin{equation}\label{HL-bound-2}
\begin{aligned}
\|\mathcal{H}(\partial_X ^{l}U)+\mathcal{L}(\partial_X ^{l}U)\|_{L^\infty} \leq C(M) \qquad \mathrm{for}\  1 \leq l \leq 4. 
\end{aligned}
\end{equation}
\end{corollary}

Since $\mathcal{H}(\partial_{X}U)$ does not provide any temporal decay, we will use the spatial decay of $\mathcal{H}(\partial_{X}U)$ in the middle field $h\leq |X|\leq \frac{1}{2}e^{\frac{3s}{2}}$ to kill the (partial) growth of the weight $\langle X\rangle^{\frac{2}{3}}$ to close the bootstrap assumption
\(\eqref{eq:5.1}_2\). 

\begin{lemma}\label{FM-5}  For $(X,s)$ such that $h\leq |X|\leq \frac{1}{2}e^{\frac{3s}{2}}$, we have
\begin{equation}\label{spatial decay}
\begin{aligned}
				|\mathcal{H}(\partial_{X}U)(X,s)| \leq \left\{\begin{array}{ll}
					Ce^{\frac{3}{2}\alpha s}\langle X\rangle^{-\frac{1}{2}-\alpha}+Ce^{-\frac{3s}{4}} &\quad \mathrm{for}\  -1/2<\alpha<0, \\
					Ce^{-\frac{5s}{8}} &  \quad \mathrm{for}\ \alpha\leq-1/2.
				\end{array}\right.
\end{aligned}
\end{equation}

\end{lemma}

\begin{proof} By \eqref{integral form}, we split the integral into different parts in the following way: 
\begin{equation*}
\begin{aligned}
&\mathcal{H}(\partial_{X}U)(X,s)\\
&=e^{-\frac{3s}{2}}\int_{A_1}\mathrm{sign}(X-Y)G_{\alpha}\big(e^{-\frac{3s}{2}}(X-Y)\big)\big(\partial_{X}U(Y,s)-\partial_{X}U(X,s)\big)\, d Y\\
&\quad+e^{-\frac{3s}{2}}\bigg(\int_{A_2}+\int_{A_3}+\int_{A_4}\bigg)\mathrm{sign}(X-Y)G_{\alpha}\big(e^{-\frac{3s}{2}}(X-Y)\big)\partial_{X}U(Y,s)\, d Y\\
&=: I_1+I_2+I_3+I_4,
\end{aligned}
\end{equation*}
where 
\begin{equation*}
\begin{aligned}
&A_1=\{|X-Y|\leq \langle X\rangle^{-\frac{2}{3}}\},\\
&A_2=\{\langle X\rangle^{-\frac{2}{3}}<|X-Y|\leq 1\},\\
&A_3=\{1<|X-Y|\leq e^{\frac{3s}{2}}\},\\
&A_4=\{|X-Y|>e^{\frac{3s}{2}}\}.
\end{aligned}
\end{equation*}

Notice that $G_{\alpha}$ is more singular when $\alpha<0$ is larger when $|X|\leq 1$. It suffices to consider the case $-\frac{1}{2}<\alpha<0$ since the other cases can be handled more easily.

For $I_1$, one may use the mean value theorem to bound it as follows:
\begin{equation}\label{sd-1}
\begin{aligned}
|I_1|\leq e^{\frac{3}{2}\alpha s}\|\partial_{X}^2U\|_{L^\infty}\int_{A_1}|X-Y|^{-\alpha}\, d Y \leq C(M)e^{\frac{3}{2}\alpha s}\langle X\rangle^{-\frac{2}{3}(1-\alpha)}.
\end{aligned}
\end{equation}

Considering $I_2$, we divide it into two cases:\\
(I) $|X|\leq \frac{1}{2}e^{\frac{3s}{2}}-1$. In this case, since $|X-Y|\leq 1$, one has
\begin{equation*}
\begin{aligned}
|Y|\leq |X-Y|+|X|\leq \frac{1}{2}e^{\frac{3s}{2}},
\end{aligned}
\end{equation*}
which together with \(\eqref{eq:5.1}_2\) gives 
\begin{equation*}
\begin{aligned}
|\partial_{X}U(Y,s)|\leq (\varepsilon^{\frac{1}{8}}+1)\langle Y\rangle^{-\frac{2}{3}}\lesssim \langle X\rangle^{-\frac{2}{3}}.
\end{aligned}
\end{equation*}
This implies that
\begin{equation}\label{sd-2}
\begin{aligned}
|I_2|\lesssim e^{\frac{3}{2}\alpha s}\int_{A_2}\frac{\langle X\rangle^{-\frac{2}{3}}}{|X-Y|^{1+\alpha}}\, d Y \lesssim e^{\frac{3}{2}\alpha s}\langle X\rangle^{-\frac{2}{3}}.
\end{aligned}
\end{equation}
(II) $\frac{1}{2}e^{\frac{3s}{2}}-1\leq |X|\leq \frac{1}{2}e^{\frac{3s}{2}}$. It suffices to consider the situation $|Y|>\frac{1}{2}e^{\frac{3s}{2}}$, otherwise we are back to (I) and have \eqref{sd-2}. 
Suppose $|Y|>\frac{1}{2}e^{\frac{3s}{2}}$, then it follows from \(\eqref{eq:5.1}_3\) that
\begin{equation*}
\begin{aligned}
|\partial_{X}U(Y,s)|\leq 2e^{-s}\lesssim |X|^{-\frac{2}{3}}\lesssim \langle X\rangle^{-\frac{2}{3}}.
\end{aligned}
\end{equation*}
Again \eqref{sd-2} holds.

To tackle $I_3$, we first assume $|X|\geq 3$ and only analyze $X\geq 3$ since $X\leq -3$ can be handed similarly. 
To this end, we decompose it in the following manner:
\begin{equation*}
\begin{aligned}
I_3&=e^{-\frac{3s}{2}}\bigg(\int_{A_{3_1}}+\int_{A_{3_2}}+\int_{A_{3_3}}\bigg)\mathrm{sign}(X-Y)G_{\alpha}\big(e^{-\frac{3s}{2}}(X-Y)\big)\partial_{X}U(Y,s)\, d Y\\
&=: I_{3_1}+I_{3_2}+I_{3_3},
\end{aligned}
\end{equation*}
where 
\begin{equation*}
\begin{aligned}
&A_{3_1}=\{X/2<|X-Y|\leq e^{\frac{3s}{2}}\},\\
&A_{3_2}=\{X/2\leq Y\leq X-1\},\\
&A_{3_3}=\{X+1\leq Y\leq 3X/2\}.
\end{aligned}
\end{equation*}

First of all, one may estimate $I_{3_1}$ and $I_{3_2}$ as follows:
\begin{equation}\label{sd-2.1}
\begin{aligned}
|I_{3_1}|&\leq e^{\frac{3}{2}\alpha s}\|\partial_{X}U\|_{L^2}\bigg(\int_{A_{3_1}}\frac{1}{|X-Y|^{2(1+\alpha)}}\, d Y\bigg)^{1/2} \\
&\lesssim e^{\frac{3}{2}\alpha s}\bigg| |Z|^{-\frac{1}{2}-\alpha}\big|_{\frac{X}{2}}^{e^{\frac{3s}{2}}}\bigg|\lesssim e^{\frac{3}{2}\alpha s}\langle X\rangle^{-\frac{1}{2}-\alpha}
\end{aligned}
\end{equation}
and
\begin{equation}\label{sd-2.2}
\begin{aligned}
|I_{3_2}|&\lesssim e^{\frac{3}{2}\alpha s} \int_{A_{3_2}}\frac{\langle Y\rangle^{-\frac{2}{3}}}{|X-Y|^{1+\alpha}}\, d Y \lesssim e^{\frac{3}{2}\alpha s}\langle X\rangle^{-\frac{2}{3}} \int_{A_{3_2}}\frac{1}{|X-Y|^{1+\alpha}}\, d Y\\
&\lesssim e^{\frac{3}{2}\alpha s}\langle X\rangle^{-\frac{2}{3}}\bigg| |Z|^{-\alpha}\big|_{1}^{\frac{X}{2}}\bigg|
\lesssim e^{\frac{3}{2}\alpha s}\langle X\rangle^{-\frac{2}{3}-\alpha},
\end{aligned}
\end{equation}
where \eqref{eq:7.1} and \(\eqref{eq:5.1}_2\) have been used. 

Second, for $I_{3_3}$, we divide it into two cases:\\
(I) When $3<X\leq \frac{1}{3}e^{\frac{3s}{2}}$, that is, $\frac{X}{2}\geq 1$ and $\frac{3X}{2}\leq \frac{1}{2}e^{\frac{3s}{2}}$. Then one has
\begin{equation}\label{sd-2.3}
\begin{aligned}
|I_{3_3}|\lesssim e^{\frac{3}{2}\alpha s} \int_{A_{3_3}}\frac{\langle Y\rangle^{-\frac{2}{3}}}{|X-Y|^{1+\alpha}}\, d Y\lesssim e^{\frac{3}{2}\alpha s}\langle X\rangle^{-\frac{2}{3}-\alpha}.
\end{aligned}
\end{equation}
(II) When $\frac{1}{3}e^{\frac{3s}{2}}<X\leq \frac{1}{2}e^{\frac{3s}{2}}$, that is, $\frac{3X}{2}> \frac{1}{2}e^{\frac{3s}{2}}$. Then one may estimate
\begin{equation}\label{sd-2.4}
\begin{aligned}
|I_{3_3}|&\leq e^{-\frac{3s}{2}}\bigg( \int_{\min(X+1, \frac{1}{2}e^{\frac{3s}{2}})}^{\frac{1}{2}e^{\frac{3s}{2}}}+\int_{\max(X+1, \frac{1}{2}e^{\frac{3s}{2}})}^{\frac{3X}{2}}
\bigg)\\
&\quad\times G_{\alpha}\big(e^{-\frac{3s}{2}}(X-Y)\big)|\partial_{X}U(Y,s)|\, d Y \\
&\lesssim e^{\frac{3}{2}\alpha s} \int_{X+1}^{\frac{3X}{2}}\frac{\langle Y\rangle^{-\frac{2}{3}}}{|X-Y|^{1+\alpha}}\, d Y
+e^{\frac{3}{2}\alpha s} \int_{X+1}^{\frac{3X}{2}}\frac{e^{-s}}{|X-Y|^{1+\alpha}}\, d Y\\
&\lesssim e^{\frac{3}{2}\alpha s}\langle X\rangle^{-\frac{2}{3}-\alpha}+e^{\frac{3}{2}\alpha s}e^{-s}\langle X\rangle^{-\alpha}\lesssim e^{\frac{3}{2}\alpha s}\langle X\rangle^{-\frac{2}{3}-\alpha}.
\end{aligned}
\end{equation}

Last, it remains to consider the case $h\leq |X|<3$. Indeed, noticing $\langle X\rangle \sim 1$, one can  simply estimate  
\begin{equation}\label{sd-2.5}
\begin{aligned}
|I_{3_3}|&\leq e^{\frac{3}{2}\alpha s}\|\partial_{X}U\|_{L^2}\bigg(\int_{A_3}\frac{1}{|X-Y|^{2(1+\alpha)}}\, d Y\bigg)^{1/2} \\
&\lesssim e^{\frac{3}{2}\alpha s}\lesssim e^{\frac{3}{2}\alpha s}\langle X\rangle^{-\frac{1}{2}-\alpha}.
\end{aligned}
\end{equation}

We finally treat with $I_4$. Since $G_{\alpha}$ decays exponentially in space when when $|X|> 1$, this will contributes a good temporal decay for $I_4$. Indeed, we have
\begin{equation}\label{sd-3}
\begin{aligned}
|I_4|\leq e^{-\frac{3s}{2}}\|\partial_{X}U\|_{L^2}\bigg(\int_{A_3}e^{-e^{-\frac{3s}{2}}|X-Y|}\, d Y\bigg)^{\frac{1}{2}}\lesssim e^{-\frac{3s}{2}}e^{\frac{3s}{4}} \lesssim e^{-\frac{3s}{4}}. 
\end{aligned}
\end{equation}

Collecting \eqref{sd-1}-\eqref{sd-3}  yields \eqref{spatial decay}.

\end{proof}

To close the bootstrap assumption
\(\eqref{eq:5.1}_3\), one needs the following temporal decay estimate for $\mathcal{H}(\partial_{X}U)$ in the far field $\frac{1}{2}e^{\frac{3s}{2}}\leq |X|<\infty $. 

\begin{lemma}\label{FM-6}  For $(X,s)$ such that $\frac{1}{2}e^{\frac{3s}{2}}\leq |X|<\infty $, we have
\begin{equation}\label{spatial decay-2}
\begin{aligned}
|\mathcal{H}(\partial_{X}U)(X,s)|\leq Ce^{-\frac{5s}{8}}.
\end{aligned}
\end{equation}

\end{lemma}

\begin{proof}

We instead decompose the integral as follows: 
\begin{equation*}
\begin{aligned}
&\mathcal{H}(\partial_{X}U)(X,s)\\
&=e^{-\frac{3s}{2}}\int_{B_1}\mathrm{sign}(X-Y)G_{\alpha}\big(e^{-\frac{3s}{2}}(X-Y)\big)\big(\partial_{X}U(Y,s)-\partial_{X}U(X,s)\big)\, d Y\\
&\quad+e^{-\frac{3s}{2}}\bigg(\int_{B_2}+\int_{B_3}+\int_{B_4}\bigg)\mathrm{sign}(X-Y)G_{\alpha}\big(e^{-\frac{3s}{2}}(X-Y)\big)\partial_{X}U(Y,s)\, d Y\\
&=: J_1+J_2+J_3+J_4,
\end{aligned}
\end{equation*}
where 
\begin{equation*}
\begin{aligned}
&B_1=\{|X-Y|\leq e^{-s}\},\\
&B_2=\{e^{-s}<|X-Y|\leq 1\},\\
&B_3=\{1<|X-Y|\leq e^{\frac{3s}{2}}\},\\
&B_4=\{|X-Y|>e^{\frac{3s}{2}}\}.
\end{aligned}
\end{equation*}

Again we only consider the case $-\frac{1}{2}<\alpha<0$ since the other cases can be handled more easily.

Notice that $J_4$ contributes a bound $e^{-\frac{3s}{4}}$ which can be estimated exactly as $I_4$ in \eqref{sd-3}. 
For $J_1$, similar to \eqref{sd-1}, it holds that
\begin{equation*}
\begin{aligned}
|J_1|\leq e^{\frac{3}{2}\alpha s}\|\partial_{X}^2U\|_{L^\infty}\int_{B_1}|X-Y|^{-\alpha}\, d Y \leq C(M)e^{(\frac{5}{2}\alpha-1) s}.
\end{aligned}
\end{equation*}

To handle $J_2$, again there are two cases to consider:\\
(i) $|X|\geq \frac{1}{2}e^{\frac{3s}{2}}+1$. Since
\begin{equation*}
\begin{aligned}
|Y|\geq |X|-|X-Y|\geq |X|-1\geq \frac{1}{2}e^{\frac{3s}{2}},
\end{aligned}
\end{equation*}
it follows from \(\eqref{eq:5.1}_3\) that
\begin{equation*}
\begin{aligned}
|\partial_{X}U(Y,s)|\leq 2e^{-s},
\end{aligned}
\end{equation*}
which leads to
\begin{equation*}
\begin{aligned}
|J_2|\lesssim e^{\frac{3}{2}\alpha s}\int_{B_2}\frac{e^{-s}}{|X-Y|^{1+\alpha}}\, d Y \lesssim e^{(\frac{3}{2}\alpha-1) s}.
\end{aligned}
\end{equation*}
(ii) $\frac{1}{2}e^{\frac{3s}{2}}\leq |X|\leq \frac{1}{2}e^{\frac{3s}{2}}+1$. Again it suffices to consider the situation $|Y|\leq \frac{1}{2}e^{\frac{3s}{2}}$. Hence, it follows from \(\eqref{eq:5.1}_2\) that
\begin{equation*}
\begin{aligned}
|\partial_{X}U(Y,s)|\leq (\varepsilon^{\frac{1}{8}}+1)\langle Y\rangle^{-\frac{2}{3}}\lesssim \langle X\rangle^{-\frac{2}{3}}\sim \bigg\langle \frac{1}{2}e^{\frac{3s}{2}}\bigg\rangle^{-\frac{2}{3}} \lesssim e^{-s}.
\end{aligned}
\end{equation*}
Then one still has the same estimate for $J_2$ in (i).

Now we deal with $J_3$. It suffices to analyze $X\geq \frac{1}{2}e^{\frac{3s}{2}}$ since $X\leq -\frac{1}{2}e^{\frac{3s}{2}}$ can be handed similarly. 
For this, we split $J_3$ into three parts:
\begin{equation*}
\begin{aligned}
J_3&=e^{-\frac{3s}{2}}\bigg(\int_{B_{3_1}}+\int_{B_{3_2}}+\int_{B_{3_3}}\bigg)\mathrm{sign}(X-Y)G_{\alpha}\big(e^{-\frac{3s}{2}}(X-Y)\big)\partial_{X}U(Y,s)\, d Y\\
&=: J_{3_1}+J_{3_2}+J_{3_3},
\end{aligned}
\end{equation*}
where 
\begin{equation*}
\begin{aligned}
&B_{3_1}=\{ e^{\frac{3s}{2}}/4< |X-Y|\leq e^{\frac{3s}{2}}\},\\
&B_{3_2}=\{X+1\leq Y\leq X+e^{\frac{3s}{2}}/4\},\\
&B_{3_3}=\{X-e^{\frac{3s}{2}}/4\leq Y\leq X-1\}.
\end{aligned}
\end{equation*}

First, one has 
\begin{equation*}
\begin{aligned}
|J_{3_1}|&\leq e^{\frac{3}{2}\alpha s}\|\partial_{X}U\|_{L^2}\bigg(\int_{B_{3_1}}\frac{1}{|X-Y|^{2(1+\alpha)}}\, d Y\bigg)^{1/2} \\
&\lesssim e^{\frac{3}{2}\alpha s}\bigg| |Z|^{-\frac{1}{2}-\alpha}\big|_{e^{\frac{3s}{2}}/4}^{e^{\frac{3s}{2}}}\bigg|\lesssim e^{-\frac{3s}{4}}
\end{aligned}
\end{equation*}
and
\begin{equation*}
\begin{aligned}
|J_{3_2}|\lesssim e^{\frac{3}{2}\alpha s} \int_{B_{3_2}}\frac{e^{-s}}{|X-Y|^{1+\alpha}}\, d Y
\lesssim e^{\frac{3}{2}\alpha s}e^{-s}\bigg| |Z|^{-\alpha}\big|_{1}^{e^{\frac{3s}{2}}/4}\bigg|
\lesssim e^{-s}.
\end{aligned}
\end{equation*}

Second, to estimate $J_{3_3}$, there are two cases to consider:\\
(I) When $X\geq \frac{3}{4}e^{\frac{3s}{2}}$. Then one finds that
\begin{equation*}
\begin{aligned}
|J_{3_3}|\lesssim e^{\frac{3}{2}\alpha s} \int_{B_{3_3}}\frac{e^{-s}}{|X-Y|^{1+\alpha}}\, d Y
\lesssim e^{\frac{3}{2}\alpha s}e^{-s}\bigg| |Z|^{-\alpha}\big|_{1}^{e^{\frac{3s}{2}}/4}\bigg|\lesssim e^{-s}.
\end{aligned}
\end{equation*}
(II) When $\frac{1}{2}e^{\frac{3s}{2}}\leq X< \frac{3}{4}e^{\frac{3s}{2}}$. Then one deduces that
\begin{equation*}
\begin{aligned}
|J_{3_3}|&\leq e^{-\frac{3s}{2}}\bigg( \int_{X-\frac{1}{4}e^{\frac{3s}{2}}}^{\min(X-1, \frac{1}{2}e^{\frac{3s}{2}})}+\int_{\min(X-1, \frac{1}{2}e^{\frac{3s}{2}})}^{X-1}
\bigg)\\
&\quad\times G_{\alpha}\big(e^{-\frac{3s}{2}}(X-Y)\big)|\partial_{X}U(Y,s)|\, d Y \\
&\lesssim e^{\frac{3}{2}\alpha s}\int_{X-\frac{1}{4}e^{\frac{3s}{2}}}^{\min(X-1, \frac{1}{2}e^{\frac{3s}{2}})}\frac{\langle Y\rangle^{-\frac{2}{3}}}{|X-Y|^{1+\alpha}}\, d Y\\
&\quad+e^{\frac{3}{2}\alpha s} \int_{\min(X-1, \frac{1}{2}e^{\frac{3s}{2}})}^{X-1}\frac{e^{-s}}{|X-Y|^{1+\alpha}}\, d Y\\
&\lesssim e^{\frac{3}{2}\alpha s}e^{-s}e^{-\frac{3}{2}\alpha s}\lesssim e^{-s}.
\end{aligned}
\end{equation*}

Collecting all the cases above, one obtains \eqref{spatial decay-2}.

\end{proof}

\section{$L^\infty$ estimates.}\label{PE}

\subsection{$L^\infty$ bound on $U+e^{\frac{s}{2}}\kappa$.}
\begin{lemma}
		It holds that
		\begin{equation}\label{eq:9.3}
           \begin{aligned}
		\|U+e^{\frac{s}{2}}\kappa\|_{L^{\infty}} \leq  Me^{\frac{s}{2}}. 
           \end{aligned}
		\end{equation}
\end{lemma}

\begin{proof}
A small calculation finds that $e^{-\frac{s}{2}}U+\kappa$ obeys the following equation:
           \begin{equation}\label{eq:9.5}
           \begin{aligned}
		\partial_s(e^{-\frac{s}{2}}U+\kappa)+V\partial_X(e^{-\frac{s}{2}}U+\kappa)
=\beta_{\tau}e^{-\frac{3s}{2}}\big(\mathcal{H}(U)+\mathcal{L}(U)\big). 
           \end{aligned}
		\end{equation}
Notice that the forcing term can be bounded as follows:
          \begin{equation}\label{eq:9.6}
           \begin{aligned}
\beta_{\tau}e^{-\frac{3s}{2}}|\mathcal{H}(U)+\mathcal{L}(U)|\leq C(M)e^{-\frac{3s}{2}}e^{\frac{s}{2}}\leq e^{-\frac{s}{2}},
           \end{aligned}
		\end{equation}
where \eqref{HL-bound-1} has been used.
We compose \eqref{eq:9.5} with the Lagrangian trajectory $\Phi^{X_{0}}(s)$ (in Appendix \ref{APP}) and use \eqref{eq:9.6} to have 
	\begin{equation*}
	\begin{aligned}
	|(e^{-\frac{s}{2}}U+\kappa)\circ \Phi^{X_{0}}(s)|&\leq |e^{-\frac{s_0}{2}}U(X_0,s_0)+\kappa(-\varepsilon)|
	+\int_{s_0}^{s}e^{-\frac{s^{\prime}}{2}}\, ds^{\prime} \\
	& \leq  \frac{1}{2}M+2\big(e^{-\frac{s_{0}}{2}}-e^{-\frac{s}{2}}\big)  
	\leq   \frac{3}{4}M. 
	\end{aligned}  
	\end{equation*}
This proves \eqref{eq:9.3} and closes the bootstrap assumption \(\eqref{eq:5.4}_2\).

\end{proof}

\subsection{Point estimates on $\tau$, $\xi$ and $\kappa$.}
	
(I) Closing \eqref{eq:5.6}. It follows from \eqref{HL-bound-2} that 
	\begin{equation}\label{eq:10}
      \begin{aligned}
	|\dot{\tau}|&=e^{-s}|\big(\mathcal{H}(\partial_XU)+\mathcal{L}(\partial_XU)\big)(0,s)| \\
&\leq   C(M)e^{-s} \leq \frac{1}{2}e^{-\frac{3s}{4}},
      \end{aligned}
	\end{equation}
where we use a power of $e^{-s}$ to absorb the $M$ factor in the last term. 
\eqref{eq:10} together with \eqref{eq:2.1} by applying the fundamental theorem of calculus yields 
     \begin{equation*}
      \begin{aligned}
	|\tau(t)| \leq  |\tau(-\varepsilon)|+ \int_{-\varepsilon}^{t}|\dot{\tau}(t^{\prime})|\, d t^{\prime} \leq   \frac{1}{2}\varepsilon^{\frac{3}{4}} (2\varepsilon^{\frac{7}{4}}+\varepsilon) \leq   \varepsilon^{\frac{7}{4}}. 
      \end{aligned}
	\end{equation*}

Now, we can show that $T_*$ is the unique fixed point of $\tau$ in $[-\varepsilon, T_*]$. Indeed, setting $f(t)=t-\tau(t)$, one notices by \eqref{eq:2.1}, \eqref{eq:2.2} and \eqref{eq:10} that
\begin{equation}\label{eq:10.1}
      \begin{aligned}
	&f(-\varepsilon)=-\varepsilon,\quad f(T_*)=0,\\
     &\dot{f}\geq 1-\frac{1}{2}e^{-\frac{3s}{4}}\geq 1-\frac{1}{2}\varepsilon^{\frac{3}{4}}>0,
      \end{aligned}
	\end{equation}
which yields $f(t)<0$, namely $t<\tau(t)$, in $[-\varepsilon, T_*)$.  

On the one hand, by \(\eqref{eq:10.1}_1\), one has
\begin{equation*}
\begin{aligned}
\int_{-\varepsilon}^{T_*}(1-\dot{\tau}(t))\, d t=f(T_*)-f(-\varepsilon)=\varepsilon. 
\end{aligned}
\end{equation*}
On the other hand, by \(\eqref{eq:10.1}_2\), it holds that 
\begin{equation*}
\begin{aligned}
\int_{-\varepsilon}^{T_*}(1-\dot{\tau}(t))\, d t\geq \int_{-\varepsilon}^{T_*}(1-\frac{1}{2}\varepsilon^{\frac{3}{4}})\, d t
=(1-\frac{1}{2}\varepsilon^{\frac{3}{4}})(T_*+\varepsilon). 
\end{aligned}
\end{equation*}
Hence, it follows that
\begin{equation}\label{eq:10.2}
\begin{aligned}
T_*\leq \frac{\varepsilon}{1-\frac{1}{2}\varepsilon^{\frac{3}{4}}}-\varepsilon\leq \varepsilon^{\frac{7}{4}}. 
\end{aligned}
\end{equation}

(II) Closing \eqref{eq:5.7}. First, one observes from \eqref{eq:9.3} that
     \begin{equation}\label{eq:10.3}
	\begin{aligned}
	|\kappa(t)|=|u(\xi(t), t)|\leq \|e^{-\frac{s}{2}}U+\kappa\|_{L^{\infty}} \leq  M,
	\end{aligned}  
	\end{equation}
and then applies \eqref{eq:6.3} and \eqref{HL-bound-2} to estimate
	\begin{equation*}
	\begin{aligned}
	|\dot{\xi}| & \leq  |\kappa|+\frac{e^{-\frac{3s}{2}}}{\left|\partial_{X}^{3} U(0,s)\right|}|\big(\mathcal{H}(\partial_X^2U)+\mathcal{L}(\partial_X^2U)\big)(0,s)| \\
	& \leq  M+C(M)e^{-\frac{3s}{2}}\leq   M+e^{-s} \leq   \frac{5}{4} M.
	\end{aligned}  
	\end{equation*}
This together with \eqref{eq:2.1} and \eqref{eq:10.2} by applying the fundamental theorem of calculus
produces  
    \begin{equation*}
	\begin{aligned}
	|\xi(t)| \leq  |\xi(-\varepsilon)|+\int_{-\varepsilon}^{t}|\dot{\xi}|\, d t^{\prime} \leq   \frac{5}{4} M\big(\varepsilon^{\frac{7}{4}}+\varepsilon\big) \leq   \frac{5}{2} M \varepsilon. 
	\end{aligned}  
	\end{equation*}
Hence, one has  
     \begin{equation}\label{eq:10.6}
	\begin{aligned}
|x_*|=|\xi(T_*)|\leq \frac{5}{2}M\varepsilon.
     \end{aligned}  
	\end{equation}

\subsection{Point estimates on $e^{\frac{s}{2}}(\kappa-\dot{\xi})$ and $e^{-\frac{s}{2}}\dot{\kappa}$.} It 
follows from \eqref{eq:3.5} and \eqref{eq:3.6} that
		\begin{equation}\label{eq:10.7}
           \begin{aligned}
		e^{\frac{s}{2}}|\kappa-\dot{\xi}|
&\leq \frac{e^{-s}}{|\partial_{X}^{3}U(0,s)|}
|\big(\mathcal{H}(\partial_X^2U)+\mathcal{L}(\partial_X^2U)\big)(0,s)|\\
&\leq C(M)e^{-s} \leq  e^{-\frac{3s}{4}},
           \end{aligned}
		\end{equation}
and 
\begin{equation}\label{eq:10.8}
\begin{aligned}
e^{-\frac{s}{2}}|\dot{\kappa}|
&\leq e^{\frac{s}{2}}|\kappa-\dot{\xi}|+e^{-s}|\big(\mathcal{H}(U)+\mathcal{L}(U)\big)(0,s)|\\
&\leq e^{-\frac{3s}{4}}+C(M)e^{-s}e^{\frac{s}{2}}\leq e^{-\frac{s}{3}}. 
\end{aligned}
\end{equation}
Here one has used \eqref{HL-bound-2} and \eqref{HL-bound-1} in \eqref{eq:10.7} and \eqref{eq:10.8}, respectively. 

\subsection{$L^\infty$ bound on $U$.}	 It follows from \eqref{eq:9.3} and \eqref{eq:10.3}
that
\begin{equation}\label{U-bound}
\begin{aligned}
\|U\|_{L^\infty}\leq \|U+e^{\frac{s}{2}}\kappa\|_{L^{\infty}}+\|e^{\frac{s}{2}}\kappa\|_{L^{\infty}} \leq  2Me^{\frac{s}{2}}.
\end{aligned}
\end{equation}

\subsection{$L^\infty$ bound on $\partial_{X}^{4}U$.}
Recall \eqref{eq:2.6} with $n=4$: 
	\begin{equation}\label{eq:11}
	\begin{aligned}
	&\bigg(\partial_{s}+\underbrace{\frac{11}{2}+5 \beta_{\tau} \partial_{X} U}_{\mathcal{D}}\bigg) \partial_{X}^{4} U+V\partial_{X}^{5} U \\
	&=\underbrace{\beta_{\tau}e^{-s}\big(\mathcal{H}(\partial_{X}^{4}U)+\mathcal{L}(\partial_{X}^{4}U)\big)-10 \beta_{\tau} \partial_{X}^{2} U \partial_{X}^{3} U}_{\mathcal{F}}.
	\end{aligned}  
	\end{equation}	
We start with bounding the damping below by using \eqref{eq:5.8} and 
\eqref{eq:6.1} as follows: 
\begin{equation*}
	\begin{aligned}
\mathcal{D} \geq   \frac{11}{2}-5\times\frac{101}{100}\times\frac{101}{100} \geq   \frac{1}{4},
\end{aligned}  
	\end{equation*}
which gives
     \begin{equation}\label{eq:11.1}
	\begin{aligned}
	e^{-\int_{s_{0}}^{s} \mathcal{D}\circ\Phi^{X_{0}}(s^{\prime}) \, d s^{\prime}} \leq   e^{-\frac{1}{4}\left(s-s_{0}\right)}.
     \end{aligned} 
	\end{equation}
Then we estimate the forcing terms by \eqref{eq:5.9}, \eqref{eq:6.0}, and \eqref{HL-bound-2} as follows:
\begin{equation}\label{eq:11.2}
	\begin{aligned}
	|\mathcal{F}| &\leq C(M)e^{-s}+C\|\partial^{2}_{X}U\|_{L^{\infty}}\|\partial_{X}^{3}U\|_{L^{\infty}}\\
&\leq  e^{-\frac{s}{2}}+ CM^{\frac{4}{5}} \leq \varepsilon^{\frac{1}{2}}+ CM^{\frac{4}{5}}.
\end{aligned}  
	\end{equation}

Composing \eqref{eq:11} with the Lagrangian trajectory, using \eqref{eq:11.1} and \eqref{eq:11.2} and applying Gronwall's inequality give
	\begin{equation*}
	\begin{aligned}
	\left|\partial_{X}^{4}U \circ \Phi^{X_{0}}(s)\right| 
&\leq \|\partial_{X}^{4} U_0\|_{L^{\infty}} e^{-\int_{s_{0}}^{s} \mathcal{D}\circ\Phi^{X_{0}}(s^{\prime}) \, d s^{\prime}}\\
&\quad+\int_{s_0}^{s}|\mathcal{F}\circ \Phi^{X_{0}}(s^\prime)|e^{-\int_{s^\prime}^{s} \mathcal{D}\circ\Phi^{X_{0}}(s^{\prime\prime}) \, d s^{\prime\prime}}\,ds^{\prime}\\ 
&\leq \|\partial_{X}^{4} U_0\|_{L^{\infty}} e^{-\frac{1}{4}\left(s-s_{0}\right)}+(\varepsilon^{\frac{1}{2}}+CM^{\frac{4}{5}})\int_{s_0}^{s}e^{-\frac{1}{4}(s-s^\prime)}\,ds^{\prime}\\ 
	&\leq \|\partial_{X}^{4} U_0\|_{L^{\infty}} e^{-\frac{1}{4}\left(s-s_{0}\right)}+4(\varepsilon^{\frac{1}{2}}+CM^{\frac{4}{5}})\big(1-e^{-\frac{1}{4}\left(s-s_{0}\right)}\big)\\
	&\leq \frac{1}{2}M+CM^{\frac{4}{5}}\leq \frac{3}{4} M.
	\end{aligned}
	\end{equation*}
This closes the bootstrap assumption \(\eqref{eq:5.4}_1\).

\qed

	\subsection{Near field ($ 0 \leq  |X| \leq h$).}

	\subsubsection{$L^\infty$ bound on $\partial_{X}^{4} \widetilde{U}$.}
\eqref{eq:3.1} with $n=4$ reads as follows: 
	\begin{equation}\label{eq:11.5}
	\begin{aligned}
	\bigg(\partial_{s}+\frac{11}{2}+5 \beta_{\tau} \partial_{X} U\bigg) \partial_{X}^{4} \widetilde{U}+V\partial_{X}^{5} \widetilde{U}
	=\beta_{\tau}e^{-s}\big(\mathcal{H}(\partial_{X}^{4}U)+\mathcal{L}(\partial_{X}^{4}U)\big)-\beta_{\tau}F_{\widetilde{U}}^{(4)}
	\end{aligned}
	\end{equation}
with
      \begin{equation*}
	\begin{aligned}
	F_{\widetilde{U}}^{(4)}&=\partial_{X}^{5}\overline{U}\big(\dot{\tau}\overline{U}+e^{\frac{s}{2}}(\kappa-\dot{\xi})\big)+\dot{\tau}\sum_{k=0}^{3}\binom{4}{k}\partial_{X}^{k+1}\overline{U}\partial_{X}^{4-k}\overline{U} \\
	&\quad+\sum_{k=0}^{3}{\binom{5}{k}}\partial_{X}^{k}\widetilde{U}\partial_{X}^{5-k}\overline{U}+\sum_{k=1}^{2}\binom{4}{k}\partial_{X}^{k+1}\tilde{U}\partial_{X}^{4-k}\widetilde{U}.
	\end{aligned}
	\end{equation*}	
First notice that the estimate \eqref{eq:11.1} for the damping $\mathcal{D}$ still holds. It remains to control the local forcing term $F_{\widetilde{U}}^{(4)}$. In fact,
one may use \eqref{eq:5.0}-\eqref{eq:5.4} to estimate
	\begin{equation}\label{eq:11.7}
	\begin{aligned}
	\beta_{\tau}|F_{\widetilde{U}}^{(4)}|
	&\lesssim \langle X\rangle^{-\frac{14}{3}}(\varepsilon^{\frac{3}{4}}\langle X\rangle^{\frac{1}{3}}+\varepsilon^{\frac{3}{4}})+\varepsilon^{\frac{3}{4}} \sum_{k=0}^{3}\langle X\rangle^{\frac{2}{3}-(k+1)-(4-k)} \\
	&\quad+(\varepsilon^{\frac{1}{3}} h^{4}\langle X\rangle^{-\frac{14}{3}}+\varepsilon^{\frac{1}{3}}  
h^{3}\langle X\rangle^{-\frac{11}{3}}+\varepsilon^{\frac{1}{3}} h^{2}\langle X\rangle^{-\frac{8}{3}}+\varepsilon^{\frac{1}{3}} h\langle X\rangle^{-\frac{5}{3}})\\
	&\quad+\varepsilon^{\frac{1}{8}} h^{2}\cdot\varepsilon^{\frac{1}{8}} h
	\leq 10(\varepsilon^{\frac{3}{4}}+\varepsilon^{\frac{1}{3}} h).
	\end{aligned}  
	\end{equation}
	
We compose \eqref{eq:11.5} with the Lagrangian trajectory and use \eqref{eq:11.1} and \eqref{eq:11.7} to have
	\begin{equation*}
	\begin{aligned}
	|\partial_{X}^{4} \widetilde{U}\circ \Phi^{X_{0}}(s)|&\leq   |\partial_{X}^{4} \widetilde{U}(X_0, s_{0})|e^{-\frac{1}{4}(s-s_{0})}\\ 
	&\quad+(\varepsilon^{\frac{1}{2}}+10\varepsilon^{\frac{3}{4}}+10\varepsilon^{\frac{1}{3}} h)\int_{s_0}^{s}e^{-\frac{1}{4}(s-s^{\prime})}\, ds^{\prime} \\
	& \leq   \frac{1}{4}\varepsilon^{\frac{1}{3}}+4(\varepsilon^{\frac{1}{2}}+10\varepsilon^{\frac{3}{4}}+10\varepsilon^{\frac{1}{3}} h) 
	\leq   \frac{1}{2} \varepsilon^{\frac{1}{3}}.
	\end{aligned}  
	\end{equation*}
This closes the bootstrap assumption \(\eqref{eq:5.3}_2\).

	\subsubsection{$L^\infty$ bound on $\partial_{X}^{3} \widetilde{U}(0,s)$.} 
To calculate $\partial_{X}^{3} \widetilde{U}(0,s)$, we first plug $X=0$ into \eqref{eq:3.1} and use \eqref{eq:3.3} to find
\begin{equation*}
	\begin{aligned}
&\big(\partial_{s}-4\beta_{\tau}\dot{\tau}\big) \partial_{X}^{3}\widetilde{U}(0, s)+\beta_{\tau} e^{\frac{s}{2}}(\kappa-\dot{\xi}) \partial_{X}^{4}\widetilde{U}(0, s)\\
&=\beta_{\tau}e^{-s}\big(\mathcal{H}(\partial_{X}^{3}U)+\mathcal{L}(\partial_{X}^{3}U)\big)(0, s)+18\beta_{\tau}\dot{\tau},
\end{aligned}  
	\end{equation*}
and then estimate
	\begin{equation}\label{eq:11.8}
	\begin{aligned}
|\partial_{s} \partial_{X}^{3}\widetilde{U}(0, s)| &\leq  4\beta_{\tau}|\dot{\tau}||\partial_{X}^{3}\widetilde{U}(0, s)|+\beta_{\tau}e^{\frac{s}{2}}|\kappa-\dot{\xi}||\partial_{X}^{4}\widetilde{U}(0, s)|\\ 
&\quad+\beta_{\tau}e^{-s}|\big(\mathcal{H}(\partial_{X}^{3}U)+\mathcal{L}(\partial_{X}^{3}U)\big)(0, s)|+18\beta_{\tau}|\dot{\tau}| \\
& \lesssim \varepsilon^{\frac{1}{2}}e^{-\frac{3s}{4}}+\varepsilon^{\frac{1}{3}}e^{-\frac{3s}{4}}+C(M)e^{-s}+e^{-\frac{3s}{4}}\leq Ce^{-\frac{3s}{4}},
	\end{aligned} 
	\end{equation}
where \(\eqref{eq:5.3}_2\), \eqref{eq:5.5}, \eqref{eq:5.6} and \eqref{HL-bound-2} have been used. 
	By \eqref{eq:4.2-add}, \eqref{eq:11.8} and the fundamental theorem of calculus, we have
	\begin{equation}\label{eq:11.9}
	\begin{aligned}
	|\partial_{X}^{3} \widetilde{U}(0, s)|&\leq   |\partial_{X}^{3} \widetilde{U}(0, s_{0})|+C\int_{s_{0}}^{s}|\partial_{s} \partial_{X}^{3} \widetilde{U}(0, s^{\prime})|\, d s^{\prime} \\
	& \leq   \frac{1}{4}\varepsilon^{\frac{1}{2}} +C\varepsilon^{\frac{3}{4}}(1-e^{-\frac{3}{4}(s-s_0)})
\leq \frac{1}{2}\varepsilon^{\frac{1}{2}}.
	\end{aligned}
	\end{equation}
	Therefore we have closed the bootstrap assumption \eqref{eq:5.5}.

	\subsubsection{$L^\infty$ bound on  $\partial_{X}^{j} \widetilde{U}(X,s)$ for $j=0,1,2,3$}
We first evaluate $\partial_{X}^{j} \widetilde{U}(X,s)$. 
It follows from \eqref{eq:11.9} and the fundamental theorem of calculus that
	\begin{equation}\label{eq:12}
	\begin{aligned}
	|\partial_{X}^{3} \widetilde{U}(X,s)| & \leq  |\partial_{X}^{3} \widetilde{U}(0,s)|+\int_{0}^{X}|\partial_{X}^{4} \widetilde{U}(X^{\prime}, s)|\, d X^{\prime} \\
	& \leq   \frac{1}{2}\varepsilon^{\frac{1}{2}}+\frac{1}{2} \varepsilon^{\frac{1}{3}} h
	 \leq   \frac{7}{8} \varepsilon^{\frac{1}{3}} h,
	\end{aligned}  
	\end{equation}
which closes the bootstrap assumption \(\eqref{eq:5.3}_1\).
Next, we evaluate $\partial_{X}^{j} \widetilde{U}(X,s)$ for $j=0,1,2$. For this, one first notices that the constraints \eqref{eq:3.3} imply that
      \begin{equation*}
	\begin{aligned}
     \partial_{X}^{j} \widetilde{U}(0, s)=  0,\quad j=0,1,2
      \end{aligned}  
	\end{equation*}
and then applies \eqref{eq:12} and the fundamental theorem of calculus repeatedly to deduce 
\begin{equation*}
	\begin{aligned}
&|\partial_{X}^2 \widetilde{U}(X,s)| \leq   \int_{0}^{X}|\partial_{X}^3 \widetilde{U}(X^{\prime}, s)|\, d X^{\prime} \leq   \frac{7}{8} \varepsilon^{\frac{1}{3}} h^2,\\
&|\partial_{X} \widetilde{U}(X,s)| \leq   \int_{0}^{X}|\partial_{X}^2 \widetilde{U}(X^{\prime}, s)|\, d X^{\prime} \leq   \frac{7}{8} \varepsilon^{\frac{1}{3}} h^3,\\
&|\widetilde{U}(X,s)| \leq   \int_{0}^{X}|\partial_{X} \widetilde{U}(X^{\prime}, s)|\, d X^{\prime} \leq   \frac{7}{8} \varepsilon^{\frac{1}{3}} h^4. 
\end{aligned}
	\end{equation*}

Collecting all the estimates above closes the bootstrap assumptions \(\eqref{eq:5.0}_1\), 
\(\eqref{eq:5.1}_1\) and \(\eqref{eq:5.2}_1\). 

	\subsection{Middle field ($h\leq |X|\leq \frac{1}{2}e^{\frac{3s}{2}}$).}

	\subsubsection{$L^\infty$ bound on $\widetilde{U}$.} 
Define $W:=\langle X\rangle^{-\frac{1}{3}} \widetilde{U}$. Closing the bootstrap assumption \(\eqref{eq:5.0}_2\) reduces to show 
\begin{equation}\label{eq:12.2}
	\begin{aligned}
	|W\circ \Phi^{X_{0}}(s)| \leq \frac{3}{4}\varepsilon^{\frac{1}{4}}. 
\end{aligned} 
	\end{equation}

A small calculation finds that $W$ is governed by  
	\begin{equation}\label{eq:12.3}
	\begin{aligned}
	&\bigg(\partial_{s}\underbrace{-\frac{1}{2}+\beta_{\tau} 
		\partial_{X}\overline{U}+\frac{X}{3\langle X\rangle^{2}} V}_{\mathcal{D}}\bigg) W+V \partial_{X}W \\
	&=\langle X\rangle^{-\frac{1}{3}}\beta_{\tau}\left[e^{-s}\big(\mathcal{H}(U)+\mathcal{L}(U)\big)-e^{-\frac{s}{2}}\dot{\kappa}-F_{\tilde{U}}^{(0)}\right]
	\end{aligned}  
	\end{equation}
with
     \begin{equation*}
	\begin{aligned}
	F_{\tilde{U}}^{(0)}=\partial_{X}\overline{U}\big(\dot{\tau}\overline{U}+e^{\frac{s}{2}}(\kappa-\dot{\xi})\big).
	\end{aligned}  
	\end{equation*}
We first estimate the damping as follows:
	\begin{equation}\label{eq:12.4}
	\begin{aligned}
	\mathcal{D} &=  -\frac{1}{2}+\beta_{\tau}\partial_{X}\overline{U}+\frac{X}{3\langle X\rangle^{2}}\left(\frac{3}{2} X+\beta_{\tau} U+\beta_{\tau} e^{\frac{s}{2}}(\kappa-\dot{\xi})\right) \\
	&\geq  -\frac{1}{2}-\frac{101}{100}\langle X\rangle^{-\frac{2}{3}}+\frac{1}{2}X^2\langle X\rangle^{-2}\\
&\quad-\frac{1}{3}\times\frac{101}{100}(1+\varepsilon^{\frac{1}{4}})X\langle X\rangle^{-\frac{5}{3}}-\frac{1}{3}\times\frac{101}{100}e^{-\frac{3s}{4}} X\langle X\rangle^{-2}\\
	&\geq   -3\langle X\rangle^{-\frac{2}{3}}.
	\end{aligned} 
	\end{equation}
For the non-local forcing term, we need use \eqref{HL-bound-1} to bound it by	
	\begin{equation*}
	\begin{aligned}
	\langle X\rangle^{-\frac{1}{3}}\beta_{\tau}e^{-s}|\mathcal{H}(U)+\mathcal{L}(U)|\leq  C(M)e^{-s}e^{\frac{s}{2}}\leq e^{-\frac{s}{3}}. 
	\end{aligned} 
	\end{equation*}
For the local forcing terms, we may estimate it as
     \begin{equation*}
	\begin{aligned}
	\langle X\rangle^{-\frac{1}{3}}\beta_{\tau}|e^{-\frac{s}{2}}\dot{\kappa}+F_{\tilde{U}}^{(0)}|
	 \lesssim e^{-\frac{s}{3}}+e^{-\frac{3s}{4}}+e^{-\frac{3s}{4}}\leq Ce^{-\frac{s}{3}},
	\end{aligned} 
	\end{equation*}	
where \eqref{eq:2.9}, \eqref{eq:5.6}, \eqref{eq:10.7} and \eqref{eq:10.8} have been used.

Denote by  $s_{*}$ the first time that the Lagrangian trajectory enters the region  $h \leq  |X|\leq  \frac{1}{2}e^{\frac{3s}{2}} $. Following \cite{MR4321245}, one composes the damping with the Lagrangian trajectory and uses \eqref{eq:12.4} to bound
	\begin{equation*}
	\begin{aligned}
	\int_{s_{*}}^{s}\left\langle\Phi^{X_{0}}(s^{\prime})\right\rangle^{-\frac{2}{3}}\, ds^{\prime}&\leq\int_{s_{*}}^{s}\left\langle|X_{0}|e^{\frac{1}{5}(s^{\prime}-s_{*})}\right\rangle^{-\frac{2}{3}}\, ds^{\prime}\\
	&\leq\int_{s_0}^s\left(1+h^2e^{\frac25(s^{\prime}-s_0)}\right)^{-\frac{1}{3}}\, ds^{\prime} \\
	&\leq\int_{s_0}^{s_0+5\log\frac{1}{h}} 1\, ds^{\prime}+\int_{s_0+5\log\frac{1}{h}}^sh^{-\frac{2}{3}}e^{-\frac{2}{15}(s^{\prime}-s_0)}\, ds^{\prime} \\
	&\leq 5\log\frac{1}{h}+\frac{15}{2}\leq 10\log\frac{1}{h},
	\end{aligned}
	\end{equation*}
which leads to
     \begin{equation}\label{eq:12.8}
	\begin{aligned}
	e^{-\int_{s_{*}}^{s} \mathcal{D}\circ\Phi^{X_{0}}(s^{\prime}) \, d s^{\prime}} \leq   e^{30 \log (\frac{1}{h})}=h^{-30}.
     \end{aligned} 
	\end{equation}
By composing \eqref{eq:12.3} with the Lagrangian trajectory and using \eqref{eq:12.8}, we obtain
	\begin{equation*}
	\begin{aligned}
	|W\circ \Phi^{X_{0}}(s)| & \leq  h^{-30}\left|W\circ \Phi^{X_{0}}(s_*)\right| +Ch^{-30} \int_{s_{*}}^{s}  e^{-\frac{s^\prime}{3}}\, d s^{\prime} \\
	& \leq h^{-30}(|W\circ \Phi^{X_{0}}(s_*)|+C\varepsilon^{\frac{1}{3}}).
	\end{aligned} 
	\end{equation*}

There are only two cases occurring: \\
(I)  $h \leq   |X_{0}|\leq   \frac{1}{2}e^{\frac{3s}{2}} $ and  $s_{*}=s_{0}$.  In this case, one shall use \(\eqref{eq:3.8}_2\) to get
\begin{equation*}
	\begin{aligned}
	|W\circ \Phi^{X_{0}}(s)| \leq h^{-30}(\varepsilon^{\frac{1}{2}}+C\varepsilon^{\frac{1}{3}})\leq \frac{3}{4}\varepsilon^{\frac{1}{4}}.
\end{aligned} 
	\end{equation*}
(II) $s_{*}>s_{0}$  and  $|X_{0}|=h$. In the case,  we instead apply \(\eqref{eq:5.0}_1\) to obtain
\begin{equation*}
	\begin{aligned}
|W\circ \Phi^{X_{0}}(s)| \leq h^{-30}(\frac{1}{2}\varepsilon^{\frac{1}{3}} h^{4}\langle h\rangle^{-\frac{1}{3}}+C\varepsilon^{\frac{1}{3}})\leq \frac{3}{4}\varepsilon^{\frac{1}{4}}.
\end{aligned} 
	\end{equation*}
Collecting the above two cases yields \eqref{eq:12.2}.

	\subsubsection{$L^\infty$ bound on $\partial_{X} \widetilde{U}$}
	
Setting  $W:=\langle X\rangle^{\frac{2}{3}} \partial_{X} \widetilde{U}$, in order to close the bootstrap assumption
\(\eqref{eq:5.1}_2\), it suffices to verify
\begin{equation*}
	|W\circ \Phi^{X_{0}}(s)| \leq  \frac{3}{4}\varepsilon^{\frac{1}{8}}. 
	\end{equation*}

It is elementary to calculate that
	\begin{equation}\label{eq:13.1}
	\begin{aligned}
	&\bigg(\partial_{s}+\underbrace{1+\beta_{\tau}(\partial_{X} \widetilde{U}+2 \partial_{X}\overline{U})-\frac{2}{3} \frac{X}{\langle X\rangle^{2}}V}_{\mathcal{D}}\bigg)W+V \partial_{X}W \\
	&=\langle X\rangle^{\frac{2}{3}}\beta_{\tau}\left[e^{-s}\big(\mathcal{H}(\partial_{X}U)+\mathcal{L}(\partial_{X}U)\big)-F_{\widetilde{U}}^{(1)}\right]
	\end{aligned}
	\end{equation}
with
     \begin{equation*}
	\begin{aligned}
	F_{\widetilde{U}}^{(1)}=
\partial_{X}^{2}\overline{U}\big(\dot{\tau}\overline{U}+e^{\frac{s}{2}}(\kappa-\dot{\xi})+\widetilde{U}\big)+\dot{\tau}(\partial_{X}\overline{U})^{2}.
	\end{aligned}
	\end{equation*}
The damping can be bounded below by
     \begin{equation*}
	\begin{aligned}
	\mathcal{D} &\geq  1-\frac{101}{100}(\varepsilon^{\frac{1}{8}}+2)\langle X\rangle^{-\frac{2}{3}}-X^2\langle X\rangle^{-2}\\
&\quad-\frac{2}{3}\times\frac{101}{100}(1+\varepsilon^{\frac{1}{4}})X\langle X\rangle^{-\frac{5}{3}}-\frac{2}{3}\times\frac{101}{100}e^{-\frac{3s}{4}} X\langle X\rangle^{-2}\\
	& \geq  -3\langle X\rangle^{-\frac{2}{3}},
	\end{aligned} 
	\end{equation*}
which implies that  \eqref{eq:12.8}  still holds. For the non-local term, we shall use \eqref{spatial decay} to estimate 
when $-1/2<\alpha<0$
     \begin{equation*}
	\begin{aligned}
	&\langle X\rangle^{\frac{2}{3}}\beta_{\tau}e^{-s}|\mathcal{H}(\partial_{X}U)+\mathcal{L}(\partial_{X}U)|\\
&\lesssim  e^{-s}
(e^{\frac{3}{2}\alpha s}\langle X\rangle^{\frac{1}{6}-\alpha}+\langle X\rangle^{\frac{2}{3}}e^{-\frac{3s}{4}}+\langle X\rangle^{\frac{2}{3}}e^{-s})\\
& \lesssim  e^{-s}
\big(e^{\frac{3}{2}\alpha s}e^{\frac{s}{4}}e^{-\frac{3}{2}\alpha s}+e^s e^{-\frac{3s}{4}}+e^s e^{-s}\big)\leq Ce^{-\frac{3s}{4}}
	\end{aligned}
	\end{equation*}
and for $\alpha\leq -1/2$
 \begin{equation*}
	\begin{aligned}
	&\langle X\rangle^{\frac{2}{3}}\beta_{\tau}e^{-s}|\mathcal{H}(\partial_{X}U)+\mathcal{L}(\partial_{X}U)|\\
&\lesssim  e^{-s}
(\langle X\rangle^{\frac{2}{3}}e^{-\frac{5s}{8}}+\langle X\rangle^{\frac{2}{3}}e^{-s})\\
& \leq  e^{-s}
\big(e^se^{-\frac{5s}{8}}+e^s e^{-s}\big)\leq Ce^{-\frac{5s}{8}}.
	\end{aligned}
	\end{equation*}
For the local term, we shall carefully use the decay of $\partial_{X}^j\overline{U}\ (j=1,2)$ in \eqref{eq:2.9} to kill the growth of $\langle X\rangle^{\frac{2}{3}}$ as follows: 
    \begin{equation*}
	\begin{aligned}
	\langle X\rangle^{\frac{2}{3}}\beta_{\tau}|F_{\widetilde{U}}^{(1)}|
	&\lesssim \langle X\rangle^{\frac{2}{3}}\big(e^{-\frac{3s}{4}}\langle X\rangle^{-\frac{4}{3}}
+e^{-\frac{3s}{4}}\langle X\rangle^{-\frac{5}{3}}
+\varepsilon^{\frac{1}{4}}\langle X\rangle^{-\frac{4}{3}}
+e^{-\frac{3s}{4}}\langle X\rangle^{-\frac{4}{3}}\big)\\
	&\leq  Ce^{-\frac{3s}{4}}+C\varepsilon^{\frac{1}{4}}\langle X\rangle^{-\frac{2}{3}}.
	\end{aligned}
	\end{equation*}

Composing \eqref{eq:13.1} with the Lagrangian trajectory and collecting all the estimates above, we get
     \begin{equation*}
	\begin{aligned}
	|W\circ \Phi^{X_{0}}(s)|& \leq  h^{-30}|W\circ \Phi^{X_{0}}(s_*)|
+Ch^{-30}\int_{s_{*}}^{s}e^{-\frac{5s^\prime}{8}}\, ds^{\prime}  \\
&\quad+C\varepsilon^{\frac{1}{4}}h^{-30}\int_{s_{*}}^{s}\langle\Phi^{X_{0}}(s^{\prime})\rangle^{-\frac{2}{3}}\, ds^{\prime} \\
	&\leq   h^{-30}(|W\circ \Phi^{X_{0}}(s_*))|+C\varepsilon^{\frac{5}{8}})
+C\varepsilon^{\frac{1}{4}}h^{-30}\log\frac{1}{h}.
	\end{aligned}
	\end{equation*}

Again, we have two cases to consider:\\
(I) if $h \leq |X_{0}|\leq \frac{1}{2}e^{\frac{3s}{2}}$  and  $s_{*}=s_{0}$, then it follows from 
\(\eqref{eq:3.9}_2\) that 
     \begin{equation*}
	\begin{aligned}
	|W\circ \Phi^{X_{0}}(s)|  &\leq    h^{-30}(\varepsilon^{\frac{1}{4}}+C\varepsilon^{\frac{5}{8}})
+Ch^{-30}\varepsilon^{\frac{1}{4}}\log\frac{1}{h}
	\leq \frac{3}{4}\varepsilon^{\frac{1}{8}}.
	\end{aligned}  
	\end{equation*}
(II) if  $s_{*}>s_{0}$  and  $|X_{0}|=h$, then one instead uses \(\eqref{eq:5.1}_1\) to find  
     \begin{equation*}
	\begin{aligned}
	|W\circ \Phi^{X_{0}}(s)| &\leq h^{-30}\bigg(\frac{1}{2}\varepsilon^{\frac{1}{3}}h^3\langle h\rangle^{\frac{2}{3}}+C\varepsilon^{\frac{5}{8}}\bigg)
+C\varepsilon^{\frac{1}{4}}h^{-30}\log\frac{1}{h}
	\leq   \frac{3}{4}\varepsilon^{\frac{1}{8}}.
	\end{aligned}
	\end{equation*}
Collecting the above two cases yields the desired result.

	\subsection{Far field ($\frac{1}{2}e^{\frac{3s}{2}}\leq |X|<\infty$)} 

Setting  $W:=e^s\partial_{X} U$, in order to close the bootstrap assumption
\(\eqref{eq:5.1}_3\), it suffices to show
\begin{equation*}
	|W\circ \Phi^{X_{0}}(s)| \leq  \frac{3}{2}. 
\end{equation*}

One may deduce from \eqref{eq:2.6} that $W$ satisfies 
     \begin{equation}\label{TW-1}
	\begin{aligned}
	(\partial_{s}+\beta_{\tau}\partial_{X}U)W+V\partial_{X}W
	=\beta_{\tau}\big(\mathcal{H}(\partial_{X}U)+\mathcal{L}(\partial_{X}U)\big). 
	\end{aligned}
	\end{equation}
By \(\eqref{eq:5.1}_3\), the damping is bounded below by
     \begin{equation*}
	\begin{aligned}
	\mathcal{D}=\beta_{\tau}\partial_{X}U\geq -\frac{101}{100}\times 2e^{-s}\geq -3e^{-s},
	\end{aligned}
	\end{equation*}
which gives
     \begin{equation}\label{TW-2}
	\begin{aligned}
	e^{-\int_{s_{*}}^{s} \mathcal{D}\circ\Phi^{X_{0}}(s^{\prime}) \, d s^{\prime}} \leq   e^{3e^{-s_{*}}}\leq e^{3\varepsilon}\leq \frac{5}{4}.
     \end{aligned} 
	\end{equation}
For the forcing term, one shall apply \eqref{DL} and \eqref{spatial decay-2} to estimate it as follows:
      \begin{equation}\label{TW-3}
	\begin{aligned}
	\beta_{\tau}|\mathcal{H}(\partial_{X}U)+\mathcal{L}(\partial_{X}U)|
\leq  Ce^{-\frac{5s}{8}}+C(M)e^{-s}\leq Ce^{-\frac{5s}{8}}.
	\end{aligned}
	\end{equation}
Then it follows from \eqref{TW-1}-\eqref{TW-3} that 
     \begin{equation*}
	\begin{aligned}
	|W\circ \Phi^{X_{0}}(s)|& \leq  \frac{5}{4}|W\circ \Phi^{X_{0}}(s_*)|
+C\int_{s_{*}}^{s}e^{-\frac{5s^\prime}{8}}\, ds^{\prime}\\
	&\leq    \frac{5}{4}|W\circ \Phi^{X_{0}}(s_*)|+C\varepsilon^{\frac{5}{8}}.
	\end{aligned}
	\end{equation*}

There are two cases to consider:\\
(I)  $\frac{1}{2}e^{\frac{3s}{2}} \leq |X_{0}|<\infty $  and  $s_{*}=s_{0}$. In this case, one uses \(\eqref{eq:3.9}_3\) to find 
     \begin{equation*}
	\begin{aligned}
	|W\circ \Phi^{X_{0}}(s)|  &\leq    \frac{5}{4}e^{s_0}\varepsilon^{-1}+C\varepsilon^{\frac{5}{8}}
	\leq \frac{3}{2}.
	\end{aligned}  
	\end{equation*}
(II) $s_{*}>s_{0}$  and  $|X_{0}|=\frac{1}{2}e^{\frac{3s_*}{2}}$. In this case, one instead uses \(\eqref{eq:5.1}_2\) to deduce  
     \begin{equation*}
	\begin{aligned}
	|W\circ \Phi^{X_{0}}(s)| &\leq \frac{5}{4}e^{s_*}\bigg(\varepsilon^{\frac{1}{8}}+\frac{7}{20}\bigg)\langle X_0\rangle^{-\frac{2}{3}}+C\varepsilon^{\frac{5}{8}}\\
&\leq \frac{5}{4}\times \frac{2}{5}\times 2^{\frac{2}{3}}e^{s_*}e^{-s_*}+C\varepsilon^{\frac{5}{8}}
\leq   \frac{3}{2},
	\end{aligned}
	\end{equation*}
where one has used \eqref{eq:2.9-add} to get a sharper estimate 
          \begin{equation*}
           \begin{aligned}
		|\partial_{X} U(X, s)|& \leq |\partial_{X} \widetilde{U}(X, s)|+|\partial_{X} \overline{U}(X)| \\
&\leq \bigg(\varepsilon^{\frac{1}{8}}+\frac{7}{20}\bigg)\langle X_0\rangle^{-\frac{2}{3}}.
           \end{aligned}
		\end{equation*}

Collecting the above two cases yields the desired result.

     \subsection{$L^\infty$ bound for $\partial_{X}^{2} U$ on $h\leq |X|<\infty$}

	We recall \eqref{eq:2.6} with $n=2$
	\begin{equation*}
	\begin{aligned}
	\bigg(\partial_{s}+\underbrace{\frac{5}{2}+3 \beta_{\tau} \partial_{X} U}_{\mathcal{D}}\bigg) \partial_{X}^{2} U+V \partial_{X}^{3}U 
	=\beta_{\tau}e^{-s}\big(\mathcal{H}(\partial_{X}^{2}U)+\mathcal{L}(\partial_{X}^{2}U)\big). 
	\end{aligned} 
	\end{equation*}
To close the bootstrap assumption \(\eqref{eq:5.2}_2\), it suffices to show
     \begin{equation*}
	\begin{aligned}
	|\partial_{X}^{2} U\circ \Phi^{X_{0}}(s)| \leq   \frac{3}{4} M^{\frac{1}{5}}.
	\end{aligned}
	\end{equation*}

The forcing term can be bounded by 
	\begin{equation*}
	\begin{aligned}
	\beta_{\tau} e^{-s}|\mathcal{H}(\partial_{X}^{2}U)+\mathcal{L}(\partial_{X}^{2}U)| \leq  C(M)e^{-s}
	\leq  e^{-\frac{s}{2}}. 
	\end{aligned}
	\end{equation*}
For the damping, there are two cases to consider depending on the size of $h$. 

\underline{\emph{Case 1: $h \leq  |X|\leq   \frac{1}{2}e^{\frac{3s}{2}}$.}} In this case, 
	the damping is bounded below by
	\begin{equation*}
	\begin{aligned}
	\mathcal{D} \geq  \frac{5}{2}-3\times\frac{101}{100}(1+\varepsilon^{\frac{1}{8}})\langle X\rangle^{-\frac{2}{3}}
	 \geq  -3\langle X\rangle^{-\frac{2}{3}},
	\end{aligned} 
	\end{equation*}
which yields the same estimate \eqref{eq:12.8}. Then we have
     \begin{equation*}
	\begin{aligned}
	|\partial_{X}^{2} U\circ \Phi^{X_{0}}(s)|& \leq  h^{-30}|\partial_{X}^{2} U\circ \Phi^{X_{0}}(s_*)|+h^{-30}\int_{s_{*}}^{s}e^{-\frac{s^{\prime}}{2}}ds^{\prime}  \\
	&\leq   h^{-30}(|\partial_{X}^{2} U\circ \Phi^{X_{0}}(s_*)|+2\varepsilon^{\frac{1}{2}}).
	\end{aligned}
	\end{equation*}
This can be further estimated as follows:\\
(I)  $h \leq   |X_{0}|\leq   \frac{1}{2}e^{\frac{3s}{2}} $ and  $s_{*}=s_{0}$.  It follows from 
\(\eqref{eq:4.0}_2\) that
     \begin{equation*}
	\begin{aligned}
	|\partial_{X}^{2} U\circ \Phi^{X_{0}}(s)| \leq    h^{-30}\big(M^{\frac{1}{5}-\delta}+2\varepsilon^{\frac{1}{2}} \big)
	 \leq   \frac{3}{4} M^{\frac{1}{5}}.
	\end{aligned}  
	\end{equation*}
(II) $s_{*}>s_{0}$  and  $|X_{0}|=h$. By \(\eqref{eq:5.2}_1\), one has 
      \begin{equation*}
	\begin{aligned}
	|\partial_{X}^{2} U\circ \Phi^{X_{0}}(s)| \leq    h^{-30} \bigg(\frac{1}{2}\varepsilon^{\frac{1}{3}} h^{2}+C\langle X\rangle^{-\frac{5}{3}}+2\varepsilon^{\frac{1}{2}}\bigg) \leq   \frac{3}{4} M^{\frac{1}{5}}.
	\end{aligned}
	\end{equation*}

\underline{\emph{Case 2: $\frac{1}{2}e^{\frac{3s}{2}}\leq |X|<\infty$.}} One uses \(\eqref{eq:5.1}_3\) to bound the damping below by
	\begin{equation*}
	\begin{aligned}
	\mathcal{D} \geq   \frac{5}{2}-3\times\frac{101}{100}\times 2 e^{-s}
	\geq   2.
	\end{aligned}
	\end{equation*}
Hence, we have
      \begin{equation*}
	\begin{aligned}
	|\partial_{X}^{2} U\circ \Phi^{X_{0}}(s)| & \leq  |\partial_{X}^{2} U\circ \Phi^{X_{0}}(s_*)|e^{-2(s-s_*)} +\int_{s_{*}}^{s} e^{-\frac{s^\prime}{2}} e^{-2(s-s^\prime)}\, d s^{\prime} \\
	& \leq   M^{\frac{1}{5}-\delta}+2\varepsilon^{\frac{1}{2}}
	 \leq   \frac{3}{4} M^{\frac{1}{5}}.
	\end{aligned} 
	\end{equation*}

\section{Proof of Theorem \ref{thm2.1}}\label{Proof of main result}

Based on the preliminaries in the previous sections, we can achieve  the proof of Theorem \ref{thm2.1}.  
The proof in this paragraph is close to \cite{MR4321245}, we include it here for sake of completeness.
	
\subsection{Precise blowup information of the solution}~\label{sec5.2}

\underline{\emph{1. The regularity of $u$ before $T_*$.}} Notice that the assumptions in 
\eqref{eq:3.7}-\eqref{eq:4.3} on the initial data in self-similar variables entail $U(\cdot, s)\in H^5(\mathbb{R})$ for all $s\geq s_0$. This fact together with the following scaling relation: 
\begin{equation*}
\begin{aligned}
\left\|\partial_{x}^{n} u\right\|_{L^{2}}=e^{\left(-\frac{5}{4}+\frac{3}{2} n\right) s}\left\|\partial_{X}^{n} U\right\|_{L^{2}},\quad n=1,2,...,5,
\end{aligned}
\end{equation*}
yields $u\in C([-\varepsilon, T_*), H^5(\mathbb{R}))$.

\underline{\emph{2. Blowup time and location.}} The blowup  time and location have been 
obtained in \eqref{eq:10.2} and \eqref{eq:10.6}, respectively.

\underline{\emph{3. $L^\infty$ bound of $u$.}} It follows from \eqref{eq:9.3} that
\begin{equation*}
\begin{aligned}
\|u(\cdot,t)\|_{L^\infty}=\|e^{-\frac{s}{2}}U+\kappa\|_{L^{\infty}} \leq  M \quad \text{for}\ t\in 
[-\varepsilon, T_*].
\end{aligned}
\end{equation*}

\underline{\emph{4. Blowup rate of $\partial_x u$.}} We first claim that
\begin{equation}\label{eq:15.1}
\begin{aligned}
\frac{1}{2}(T_{*}-t)\leq \tau(t)-t\leq 2(T_{*}-t).
\end{aligned}
\end{equation}
The first and second inequality is equivalent to  
\begin{equation}\label{eq:15.2}
\begin{aligned}
T_{*}\leq 2\tau(t)-t
\end{aligned}
\end{equation}
and
\begin{equation}\label{eq:15.3}
\begin{aligned}
\tau(t)+t\leq 2T_*
\end{aligned}
\end{equation}
respectively. Both of \eqref{eq:15.2} and \eqref{eq:15.3} can be verified by the fact that $2\tau(t)-t$ and $\tau(t)+t$ is monotone decreasing and increasing due to \eqref{eq:5.6}, respectively, along with $\tau(T_*)=T_*$.

Straightforward calculations yield 
\begin{equation}\label{eq:15.5}
\begin{aligned}
\partial_{x} u(\xi(x),t)=\frac{1}{\tau(t)-t} \partial_{X}U(0, s)=-\frac{1}{\tau(t)-t},
\end{aligned}
\end{equation}
where \eqref{eq:3.3} has been used. 

It follows from \eqref{eq:15.1} and \eqref{eq:15.5} that 
\begin{equation*}
\begin{aligned}
\frac{1}{2(T_*-t)}\leq |\partial_{x} u(\xi(x), t)|\leq \frac{2}{T_*-t}.
\end{aligned}
\end{equation*}
This means that 
\begin{equation*}
\begin{aligned}
\lim_{t \rightarrow T_*}\partial_{x}u(\xi(x),t)=-\infty.
\end{aligned}
\end{equation*}
So $\partial_{x}u$ blows up at $x_*=\xi(T_*)$.

On the other hand, noticing 
\begin{equation}\label{eq:15.6}
\begin{aligned}
\partial_{x} u(x,t)=\frac{1}{\tau(t)-t} \partial_{X}U(X,s),
\end{aligned}
\end{equation}
one deduces by \eqref{eq:5.8} and \eqref{eq:15.1} that
\begin{equation*}
\begin{aligned}
\frac{1}{3(T_*-t)}\leq \|\partial_{x} u(\cdot,t)\|_{L^\infty}\leq \frac{3}{T_*-t}.
\end{aligned}
\end{equation*}

\underline{\emph{5. The regularity of $u$ at $T_*$.}} 
We first claim that if $|x-x_{*}|>\frac{1}{2}$, then 
\begin{equation}\label{eq:15.7}
\begin{aligned}
|\partial_{x}u(x,T_*)|\leq 2.
\end{aligned}
\end{equation}
To verify \eqref{eq:15.7}, one first sees that there exists  $t_{1} \in\left[-\epsilon, T_{*}\right)$  such that 
\begin{equation*}
\begin{aligned}
|x-\xi(t)| \geq   \frac{1}{2} \quad \text{for}\ t\in [t_1, T_*],
\end{aligned}
\end{equation*}
which implies in terms of the self-similar variables that
\begin{equation*}
\begin{aligned}
|X| \geq \frac{1}{2(\tau(t)-t)^{\frac{3}{2}}}= \frac{1}{2} e^{\frac{3}{2} s} \quad \text{for}\ t\in [t_1, T_*). 
\end{aligned}
\end{equation*}
This together with \(\eqref{eq:5.1}_3\) implies
\begin{equation*}
\begin{aligned}
|\partial_{x}u(x,t)|=e^{s}|\partial_{X} U(X, s)|\leq 2 \quad \text{for}\ t\in [t_1, T_*). 
\end{aligned}
\end{equation*}
Then one can send $t$ to $T_*$ in the above inequality to obtain \eqref{eq:15.7}.

We next claim that if $0<|x-x_{*}|\leq \frac{1}{2}$, then 
\begin{equation}\label{eq:15.8}
\begin{aligned}
|\partial_{x}u(x,T_*)|\sim |x-x_{*}|^{-\frac{2}{3}}.
\end{aligned}
\end{equation}
Indeed, in this case, there exists $t_{2}$ close to $T_{*}$ in $[-\epsilon, T_{*})$ such that
\begin{equation*}
\begin{aligned}
\frac{1}{2}|x-x_{*}| \leq|x-\xi(t)| \leq |x-x_{*}|\quad \text{for}\ t\in [t_2, T_*],
\end{aligned}
\end{equation*}
which implies 
\begin{equation}\label{eq:15.9}
\begin{aligned}
\frac{1}{2}|x-x_{*}| e^{\frac{3}{2} s} \leq |X| \leq |x-x_{*}|e^{\frac{3}{2}s} \leq \frac{1}{2} e^{\frac{3}{2}s} \quad \text{for}\ t\in [t_1, T_*). 
\end{aligned}
\end{equation}
By choosing a larger  $t_{2}$  if necessary, we may also assume  $h\leq|X|$. It follows from \(\eqref{eq:5.1}_2\), \eqref{eq:15.6} and \eqref{eq:15.9} that for all $t\in [t_1, T_*)$
\begin{equation*}
\begin{aligned}
|\partial_{x} u(x, t)|\lesssim \frac{1}{\tau(t)-t}|X|^{-\frac{2}{3}} \lesssim |x-x_{*}|^{-\frac{2}{3}},
\end{aligned}
\end{equation*}
and
\begin{equation*}
\begin{aligned}
|\partial_{x} u(x, t)|\gtrsim \frac{1}{\tau(t)-t}|X|^{-\frac{2}{3}} \gtrsim |x-x_{*}|^{-\frac{2}{3}}.
\end{aligned}
\end{equation*}
Then \eqref{eq:15.8} follows by sending $t$ to $T_*$ in the above two inequalities.

We conclude from \eqref{eq:15.7} and \eqref{eq:15.8} that $ u(\cdot, T_{*})\in C^{\frac{1}{3}}(\mathbb{R})$ and $u$ has a cusp singularity at $(x_{*}, T_{*})$.

\subsection{Asymptotic Convergence to Stationary Solution}~\label{sec5.2}

{\underline{\emph{Step 1.}}}
It is easy to verify by \eqref{eq:11.8} that the limit $\nu=\lim _{s \rightarrow \infty} \partial_{X}^{3} U(0,s)$ exists.
Define
     \begin{equation*}
	\begin{aligned} 
     \widetilde{U}_{\nu}:=U-\overline{U}_{\nu}
     \end{aligned} 
	\end{equation*}
To finish \eqref{eq:4.5}, it suffices to show 
\begin{equation}\label{eq:16.1}
\begin{aligned}
\limsup _{s \rightarrow \infty}|\widetilde{U}_{\nu}(X,s)|=0 \quad \text{for}\ X\in \mathbb{R}.
\end{aligned}
\end{equation}
Recalling \eqref{eq:3.7} and \eqref{eq:4.6}, \eqref{eq:16.1} is obvious when $X=0$. Thus it remains to show \eqref{eq:16.1} for $X\neq 0$.

We begin by proving \eqref{eq:16.1} in the near field, namely $|X|\in (0, |X_0|)$ for some small $|X_0|$. 
Recalling the constrains \eqref{eq:3.7}, one can Taylor expand $\widetilde{U}_{\nu}$ as follows:
\begin{equation}\label{eq:16.2}
\begin{aligned}
\widetilde{U}_{\nu}(X, s)
=\frac{X^{3}}{6} \partial_{X}^{3} \widetilde{U}_{\nu}(0,s)+\frac{X^{4}}{24} \partial_{X}^{4} \widetilde{U}_{\nu}(Y,s)
\end{aligned}
\end{equation} 
for some $Y$ between $0$ and $X$ when $X$ is close to $0$. Next, we handle the RHS of \eqref{eq:16.2} term by term. Observe that 
\begin{equation*}
\begin{aligned}
\lim _{s \rightarrow \infty}\partial_{X}^{3} \widetilde{U}_{\nu}(0,s)=\lim _{s \rightarrow \infty}\partial_{X}^{3}U(0,s)-\nu=0.
\end{aligned}
\end{equation*} 
Hence, first fix a small $|X_0|>0$, and then fix a small $\delta\in (0,  M|X_0|)$, there exists a large enough $s_0=s_0(X_0, \delta)$ such that
\begin{equation}\label{eq:16.3}
\begin{aligned}
|\partial_{X}^{3} \widetilde{U}_{\nu}(X_0,s)|\leq \delta \quad \text{for}\ s\geq s_0. 
\end{aligned}
\end{equation} 
Notice that
\begin{equation*}
\begin{aligned}
\partial_{X}^{4}\overline{U}_{\nu}(X)=\left(\frac{\nu}{6}\right)^{\frac{3}{2}}\partial_{X}^{4}\overline{U}\left(\left(\frac{\nu}{6}\right)^{\frac{1}{2}}X\right). 
\end{aligned}
\end{equation*} 
Then one may estimate 
\begin{equation}\label{eq:16.4}
\begin{aligned}
\|\partial_{X}^{4} \widetilde{U}_{\nu}\|_{L^\infty}&\leq \|\partial_{X}^{4}U\|_{L^\infty}+
\|\partial_{X}^{4}\overline{U}_{\nu}\|_{L^\infty}\\
&\leq M+100\left(\frac{\nu}{6}\right)^{\frac{3}{2}}\leq 2M.
\end{aligned}
\end{equation} 
It follows from \eqref{eq:16.2}-\eqref{eq:16.4} that
\begin{equation}\label{eq:16.5}
\begin{aligned}
|\widetilde{U}_{\nu}(X_0, s)|
\leq \frac{\delta|X_0|^{3}}{6}+\frac{M|X_0|^{4}}{12}\quad \text{for}\ s\geq s_0. 
\end{aligned}
\end{equation}

{\underline{\emph{Step 2.}}}
Basic calculations show that $\widetilde{U}_{\nu}$ evolves in time as follows:
\begin{equation}\label{eq:16.6}
\begin{aligned}
\left(\partial_{s}-\frac{1}{2}+\partial_{X}\overline{U}_{\nu}\right) \widetilde{U}_{\nu}+\bigg(\underbrace{\frac{3}{2}X+U}_{P}\bigg) \partial_{X} \widetilde{U}_{\nu}=\beta_{\tau} F_{\widetilde{U}_{\nu}}
\end{aligned}
\end{equation}
with
\begin{equation*}
\begin{aligned}
F_{\widetilde{U}_{\nu}}=e^{-s}\big(\mathcal{H}(U)+\mathcal{L}(U)\big)-e^{-\frac{s}{2}} \dot{\kappa}-e^{\frac{s}{2}}(\kappa-\dot{\xi})\partial_{X} U-\dot{\tau} U \partial_{X} U. 
\end{aligned}
\end{equation*}

The forcing term is bounded by
\begin{equation}\label{eq:16.7}
\begin{aligned}
\beta_{\tau} |F_{\widetilde{U}_{\nu}}|\lesssim C(M)e^{-s}e^{\frac{s}{2}}+e^{-\frac{s}{3}}+e^{-\frac{3s}{4}}+Me^{-\frac{3s}{4}}e^{\frac{s}{2}}\leq e^{-\frac{s}{5}},
\end{aligned}
\end{equation}
where \eqref{U-bound} has been used. 

To estimate $\widetilde{U}_{\nu}$, we need to work with a new Lagrangian trajectory associated with $P$. To this end, we consider the following Lagrangian trajectory: 
\begin{equation}\label{eq:16.8}
\begin{aligned}
\frac{d}{d s} \Psi^{X_{0}}(s)=\frac{3}{2} \Psi^{X_{0}}(s)+P\circ \Psi^{X_{0}}(s),\quad
\Psi^{X_{0}}(s_0)=X_0.
\end{aligned}
\end{equation}
By the mean value theorem and \eqref{eq:5.8}, one obtains 
\begin{equation}\label{eq:16.9}
\begin{aligned}
|U(X,s)|\leq |U(0,s)|+\|\partial_{X}U\|_{L^\infty}|X|\leq \frac{101}{100}|X|. 
\end{aligned}
\end{equation}
It follows from \eqref{eq:16.8} and \eqref{eq:16.9} that
\begin{equation*}
\begin{aligned}
\frac{d}{d s} |\Psi^{X_{0}}(s)|^2
\geq 3|\Psi^{X_{0}}(s)|^2-2\times\frac{101}{100}|\Psi^{X_{0}}(s)|^2
\geq \frac{4}{5}|\Psi^{X_{0}}(s)|^2,
\end{aligned}
\end{equation*}
which together with $\Psi^{X_{0}}(s_0)=X_0$ gives 
\begin{equation}\label{eq:17}
\begin{aligned}
|\Psi^{X_{0}}(s)|
\geq |X_0|e^{\frac{2}{5}(s-s_{0})}.
\end{aligned}
\end{equation}

Set  $G(X, s)=e^{-\frac{3}{2}(s-s_{0})} \widetilde{U}_{\nu}(X, s)$.  By  
\eqref{eq:16.6}, it is direct to check that 
\begin{equation}\label{eq:17.1}
\begin{aligned}
\left(\frac{d}{d s}+1+\partial_{X}\overline{U}_{\nu}\right) G\circ \Psi^{X_{0}}(s)
=\beta_{\tau} e^{-\frac{3}{2}(s-s_{0})} F_{\widetilde{U}_{\nu}}\circ \Psi^{X_{0}}(s)
\end{aligned}
\end{equation}
with the damping bounded below by $0$ due to \eqref{extremum} that  
\begin{equation*}
\begin{aligned}
1+\partial_{X}\overline{U}_{\nu}\geq 1-\|\partial_{X}\overline{U}_{\nu}\|_{L^\infty}
=1-\|\partial_{X}\overline{U}\|_{L^\infty}=0.
\end{aligned}
\end{equation*}

Plugging \eqref{eq:16.5} and \eqref{eq:16.7} into \eqref{eq:17.1} and applying Gronwall's inequality yields 
\begin{equation*}
\begin{aligned}
|G\circ \Psi^{X_{0}}(s)| & \leq |G\circ \Psi^{X_{0}}(s_0)|+\beta_{\tau} \int_{s_{0}}^{s} |F_{\widetilde{U}_{\nu}}\circ \Psi^{X_{0}}(s_0)|e^{-\frac{3}{2}(s^{\prime}-s_{0})}\, d s^{\prime} \\
& \leq \frac{\delta|X_0|^{3}}{6}+\frac{M|X_0|^{4}}{12}+\int_{s_{0}}^{s} e^{-\frac{s^\prime}{5}}e^{-\frac{3}{2}(s^{\prime}-s_{0})}\, d s^{\prime} \\
& \leq \frac{\delta|X_0|^{3}}{6}+\frac{M|X_0|^{4}}{12}+\frac{M|X_0|^{4}}{12}
 \leq M|X_0|^{4},
\end{aligned}
\end{equation*}
that is
\begin{equation}\label{eq:17.2}
\begin{aligned}
|\widetilde{U}_{\nu}\circ \Psi^{X_{0}}(s)| 
 \leq M|X_0|^{4}e^{-\frac{3}{2}(s-s_{0})}.
\end{aligned}
\end{equation}

Denote $s_*=s_{0}+\frac{13}{5} \log |X_0|^{-1}$. 
Thus for $s \in (s_0, s_{*})$, it follows from \eqref{eq:17.2} that
\begin{equation}\label{eq:17.3}
\begin{aligned}
|\widetilde{U}_{\nu}\circ \Psi^{X_{0}}(s)| 
 \leq M|X_0|^{4-\frac{3}{2}\frac{13}{5}}\leq M|X_0|^{\frac{1}{10}}. 
\end{aligned}
\end{equation}
For all  $X$  between  $X_0$  and  $\Psi^{X_{0}}(s_{*})$, there exists  $s \in (s_0, s_{*})$  such that  $X=\Psi^{X_{0}}(s)$. Therefore, for such  $(X, s)$,  \eqref{eq:17.3} gives
\begin{equation}\label{eq:17.5}
\begin{aligned}
|\widetilde{U}_{\nu}(X,s)|\leq M|X_0|^{\frac{1}{10}}. 
\end{aligned}
\end{equation}
By \eqref{eq:17}, one infers that \eqref{eq:17.5} will cover at least all  $X$ satisfying
\begin{equation}\label{eq:17.6}
\begin{aligned}
|X_0|\leq |X|\leq |X_0|e^{\frac{2}{5}\cdot \frac{13}{5} \log |X_0|^{-1}}=|X_0|^{-\frac{1}{25}}.
\end{aligned}
\end{equation}

Hence, combining \eqref{eq:17.5} and \eqref{eq:17.6} and taking the limit that  $s_{0} \rightarrow \infty$, one obtains 
\begin{equation}\label{eq:17.8}
\begin{aligned}
\limsup _{s \rightarrow \infty}|\widetilde{U}_{\nu}(y, s)|\leq M|X_0|^{\frac{1}{10}}.
\end{aligned}
\end{equation}
Finally, sending  $X_0 \rightarrow 0$ in \eqref{eq:17.8}  proves  for all  $X \neq 0$ that 
\begin{equation}
\begin{aligned}
\limsup _{s \rightarrow \infty}|\widetilde{U}_{\nu}(y, s)|=0.
\end{aligned}
\end{equation}
This completes the proof of \eqref{eq:16.1}.

     \appendix 
	\section{}\label{APP}
		
	Let $\Phi^{X_0}(s)$ be the Lagrangian trajectory defined by
	\begin{equation}\label{eq:20}
     \begin{aligned}
	\frac{d}{d s}\Phi^{X_{0}}(s)=V\circ \Phi^{X_{0}}(s),\quad
\Phi^{X_{0}}(s_0)=X_0,
     \end{aligned}
	\end{equation}
where $V$ is the transport speed given by in \eqref{def:1}.

We first present an upper bound for the Lagrangian trajectory $\Phi^{X_{0}}(s)$, which shows that $\Phi^{X_{0}}(s)$ grows at most exponentially with time, but is bounded for any fixed $X_0\in \mathbb{R}$ when $s<\infty$. This means that $\Phi^{X_{0}}(s)$ is well-defined. 

	\begin{lemma}~\label{lem3.2}
		For all  $X_{0} \in \mathbb{R}$, we have
		\begin{equation}\label{eq:20.1}
           \begin{aligned}
		|\Phi^{X_{0}}(s)| \leq  \big(|X_{0}|+6M\varepsilon^{-\frac{1}{2}}\big) e^{\frac{3}{2}(s-s_{0})}.
           \end{aligned}
		\end{equation}
	\end{lemma}
	\begin{proof}[\rm \textbf{Proof}]
		It follows from \eqref{eq:20} that
		\begin{equation*}
	\begin{aligned}
		\frac{d}{d s} \left(e^{-\frac{3s}{2}}\Phi^{X_{0}}(s) \right)=\beta_{\tau} e^{-\frac{3s}{2}} \big(U\circ \Phi^{X_{0}}(s)+e^{\frac{s}{2}}(\kappa-\dot{\xi}) \big),
		\end{aligned}
	\end{equation*}
which together with \eqref{U-bound} and \eqref{eq:10.7} gives
		\begin{equation*}
		\begin{aligned}
		|\Phi^{X_{0}}(s) | & \leq  |X_{0}| e^{\frac{3}{2} (s-s_{0})}+e^{\frac{3s}{2}} \int_{s_{0}}^{s} \beta_{\tau}\big(U\circ \Phi^{X_{0}}(s^{\prime})+e^{\frac{s^{\prime}}{2}}(\kappa-\dot{\xi})\big)e^{-\frac{3s^{\prime}}{2}}\, d s^{\prime} \\
		& \leq  |X_{0} | e^{\frac{3}{2} (s-s_{0})}+ e^{\frac{3s}{2}} \int_{s_{0}}^{s} \frac{101}{100}\big(2Me^{\frac{s^{\prime}}{2}}+e^{-s^{\prime}}\big)e^{-\frac{3s^{\prime}}{2}}\, d s^{\prime}\\
&\leq \big(|X_{0}|+6M\varepsilon^{-\frac{1}{2}}\big) e^{\frac{3}{2}(s-s_{0})}.
		\end{aligned}
		\end{equation*}	
This completes the proof of \eqref{eq:20.1}.

	\end{proof}

	Next, we present a lower bound for $\Phi^{X_{0}}(s)$, which tells that $\Phi^{X_{0}}(s)$ will escape to infinity exponentially fast, once the particle $X_{0}$ is at least away from $0$ by a distance $h$. This lower bound will play a role in converting spatial decay into temporal decay.
	\begin{lemma}~\label{lem3.3}(\cite{MR4321245})
		For  $h \leq  |X_{0}|<\infty$, we have
		\begin{equation}\label{eq:20.2}
           \begin{aligned}
		|\Phi^{X_{0}}(s) |\geq  |X_{0}| e^{\frac{1}{5} (s-s_{0})}.
           \end{aligned}
		\end{equation}
	\end{lemma}
	\begin{proof} 
We a priori assume
\begin{equation}\label{eq:20.3}
\begin{aligned}
|\Phi^{X_{0}}(s)|\geq |X_{0}|.
\end{aligned}
\end{equation}
By the assumption of this lemma and the definition of $h$, it holds that
\begin{equation*}
\begin{aligned}
|X_{0}|\geq h \geq 100\varepsilon\geq 100 e^{-s}. 
\end{aligned}
\end{equation*}
Hence, one obtains
\begin{equation}\label{eq:20.5}
\begin{aligned}
e^{-s}\leq \frac{1}{100}|\Phi^{X_{0}}(s)|.
\end{aligned}
\end{equation}

By the mean value theorem and \eqref{eq:5.8}, one obtains
\begin{equation*}
\begin{aligned}
|U(X,s)|\leq |U(0,s)|+\|\partial_{X}U\|_{L^\infty}|X|\leq \frac{101}{100}|X|. 
\end{aligned}
\end{equation*}
This together with \eqref{def:1}, \eqref{eq:10.7} and \eqref{eq:20.5} yields
\begin{equation}\label{eq:20.6}
\begin{aligned}
|V\circ \Phi^{X_{0}}(s)|&\geq \frac{3}{2}|\Phi^{X_{0}}(s)|-\beta_{\tau}\big(|U\circ \Phi^{X_{0}}(s)|+e^{\frac{s}{2}}|\kappa-\dot{\xi}|\big)\\
&\geq \frac{3}{2}|\Phi^{X_{0}}(s)|-\frac{101}{100}\times \frac{101}{100}|\Phi^{X_{0}}(s)|-\frac{101}{100}e^{-s}\\
&\geq \frac{1}{5}|\Phi^{X_{0}}(s)|.
\end{aligned}
\end{equation}
It follows from \eqref{eq:20} and \eqref{eq:20.6}
\begin{equation*}
\begin{aligned}
\frac{d}{d s} |\Psi^{X_{0}}(s)|^2
\geq \frac{2}{5}|\Psi^{X_{0}}(s)|^2.
\end{aligned}
\end{equation*}
Solving this inequality gives \eqref{eq:20.2}, which also closes \eqref{eq:20.3}.

	\end{proof}

\section*{Acknowledgments}
 Saut acknowledges the support of the project ANR ISAAC AAPG2023. Wang acknowledges the support of grant no. 830018 from China.

\section*{Conflict of Interest}
The authors declare that they have no conflict of interest to this work.

\section*{Data Availability Statement}
The authors confirm that the data supporting the findings of this study are available within the article and its supplementary materials.

	
	
\end{document}